\documentclass[oneside,reqno,11pt]{amsart}
\usepackage{soul}
\usepackage{epsfig}
\usepackage{color}
\usepackage{amssymb,amsmath,amsthm,amstext,amsfonts,amscd,latexsym}
\usepackage[normalem]{ulem}
\usepackage[colorlinks=true,linkcolor=blue,urlcolor=red,citecolor=blue,backref]{hyperref}
\usepackage[active]{srcltx}
\usepackage{graphicx}
\usepackage{float}
\usepackage{tabularx,array,booktabs}
\usepackage[all,arc]{xy}
\usepackage{verbatim}
\usepackage{enumerate}
\usepackage{mathrsfs}

\newif\ifnotasligadas   \notasligadastrue
\newif\ifnotasnamargem  \notasnamargemfalse

\newtheorem{theorem}{Theorem}[section]
\newtheorem{corollary}[theorem]{Corollary}
\newtheorem{proposition}[theorem]{Proposition}
\newtheorem{lemma}[theorem]{Lemma}

\theoremstyle{definition}
\newtheorem{maintheorem}{Theorem}

\newtheorem{definition}[theorem]{Definition}

\newtheorem{notation}[theorem]{Notation}

\newtheorem{claim}[theorem]{Claim}

\newtheorem{remark}[theorem]{Remark}

\newcommand{\topp}{\operatorname{top}}

\newcommand{\vep}{\varepsilon}

\newcommand{\Z}{\mathbb{Z}}
\newcommand{\N}{\mathbb{N}}
\newcommand{\R}{\mathbb{R}}

\newcommand{\diam}{\operatorname{diam}}

\newcommand{\Sing}{\operatorname{Sing}}

\newcommand{\Leb}{\operatorname{Leb}}

\newcommand{\cH}{\mathcal{H}}

\newcommand{\cL}{\mathcal{L}}

\newcommand{\cM}{\mathcal{M}}
\newcommand{\cB}{\mathcal{B}}

\newcommand{\cT}{\mathcal{T}}
\newcommand{\cW}{\mathcal{W}}
\newcommand{\cU}{\mathcal{U}}

\newcommand{\cA}{\mathcal{A}}
\newcommand{\cR}{\mathcal{R}}

\newcommand{\Msr}[2]{\mathcal{M}_{#1}(#2)}          
\newcommand{\Mer}[2]{\mathcal{M}^{e}_{#1}(#2)}      
\newcommand{\Hor}{\mathfrak{H}}                     
\newcommand{\Mspec}{\mathfrak{M}}                   
\newcommand{\Lspec}{\mathfrak{L}}                   
\newcommand{\dimu}{\dim^u_H}
\newcommand{\dims}{\dim^s_H}

\makeatletter
\let\c@equation\c@theorem
\makeatother
\numberwithin{equation}{section}

\renewcommand{\ge}{\geqslant}
\renewcommand{\geq}{\geqslant}
\renewcommand{\le}{\leqslant}
\renewcommand{\leq}{\leqslant}

\title[Hausdorff Dimension and Flexibility for Attractors]{Hausdorff Dimension and Flexibility for Attractors \\ of Non-Uniformly Hyperbolic Flows}

\author{Elias Rego}
\address[E. Rego]{School of Mathematical Sciences, Dalian University of Technology, Dalian 116024, P.R. China}
\email{rego@dlut.edu.cn}

\author{Paulo Varandas}
\address[P. Varandas]{Department of Mathematics, University of Aveiro, 3810–193 Aveiro, Portugal.}
\email{paulo.varandas@ua.pt}


\begin{document}

\begin{abstract}

We study the Hausdorff dimension of attracting sets of
three-dimensional flows whose one-dimensional quotient dynamics is
non-uniformly expanding. Our main assumption is that the quotient map
admits a full-branch uniformly expanding Gibbs--Markov induced map.
We first prove that every such attractor contains a hyperbolic
horseshoe of Hausdorff dimension strictly larger than two. As a
consequence, after an arbitrarily small $C^2$ perturbation, the
associated dynamical Lagrange and Markov spectra have robustly
non-empty interior for an open and dense set of $C^1$-observables.
We then establish intermediate dimension properties. Every admissible
intermediate Hausdorff dimension is realized by an ergodic invariant
measure supported on the attractor. Moreover, every intermediate
stable and unstable dimension is realized by a strictly ergodic
hyperbolic subsystem.
The approach extends the dimension theory previously developed for
geometric Lorenz attractors to more general the non-uniformly hyperbolic attractors including
singular-hyperbolic attractors, Rovella attractors, and broad classes of skew-product
suspension flows modeled by intermittent and Benedicks--Carleson maps.

\end{abstract}
\maketitle

\section{Introduction}

\subsection{Motivation and background}

The invariant sets produced by chaotic dynamics are seldom smooth. Recurrent classes, the sets to which the rich dynamics is confined, are often fractal objects, and the classical notion of dimension records them only crudely, as sets of zero volume. Hausdorff dimension supplies the missing gradation. It is a real number that measures, roughly speaking, how fast the number of balls needed
to cover the set grows as the scale shrinks. On a submanifold it returns the usual dimension, but it attains non-integer values on more complicated sets, such as Cantor sets. It is invariant under bi-Lipschitz changes of coordinates, which is precisely what makes it a legitimate invariant of differentiable dynamical systems. We refer the reader to \cite{Falconer, Pe97} for a detailed exposition of the general theory of fractal geometry.

It is not immediately obvious that a fractal invariant should be relevant to
dynamics, but it turns out that Hausdorff dimension is intimately related to two
classical measurements of chaos. Hausdorff dimension quantifies the size and the
geometric complexity of an invariant set, and in the physical literature the
dimension of an attractor has long been one of the standard diagnostics of
chaotic behaviour \cite{ER85}; what makes it a dynamical rather than a merely
descriptive quantity is that it is determined by entropy and Lyapunov exponents.
For a $C^{1+\alpha}$ diffeomorphism of a surface and a hyperbolic ergodic
measure $\mu$, Young \cite{Yo82} proved a formula relating $\dim_H(\mu)$ to its
entropy and Lyapunov exponents. This formula was later extended to higher
dimensions by Ledrappier and Young \cite{LY85} and completed by Barreira, Pesin
and Schmeling \cite{BPS99}. A fractal dimension close to the dimension of the
ambient space signals that the chaotic behaviour is spread over the whole phase
space; a small dimension signals that it is confined to a thin set.

For hyperbolic dynamical systems this picture becomes computable, and more than
descriptive. Bowen \cite{Bo79} and Ruelle \cite{Ru82}, studying repellers of
conformal expanding maps, discovered that the dimension of a conformal invariant
set is the unique root $s$ of the equation $P(-s\log\|Df\|)=0$, where $P$ denotes
topological pressure, thus converting a geometric question into a problem of
thermodynamic formalism. McCluskey and Manning \cite{MM83} carried this to
horseshoes of surface diffeomorphisms, showing that the dimension splits as
$\dim_H=\dim_H^s+\dim_H^{u}$, each summand being a root of a Bowen equation
along the corresponding invariant direction.

That the resulting numbers are not merely descriptive was then made clear by the
theory of homoclinic bifurcations. For a horseshoe $\Lambda$ of a surface
diffeomorphism, hyperbolicity prevails after the first homoclinic tangency if
and only if $\dim_H(\Lambda)<1$, whereas above that threshold persistent
tangencies occur with positive density \cite{PT93, PY94, MY01, MPV01}. 

For flows the corresponding theory is younger and considerably more delicate; as
Barreira observes, for various reasons flows were initially given
much less attention than maps \cite{ BaBook}. The difficulties are
structural. First, the flow direction is neutral on regular orbits and one Lyapunov exponent
always vanishes; the corresponding direction contributes a full unit of
dimension that appears as an additive constant. Second, the natural reduction of
the problem to the Poincar\'e return map is considerably more subtle when the
return time is unbounded, which is exactly the setting of the large class of
singular vector fields. In this case the Poincar\'e map is discontinuous and has
unbounded derivative, and the standard reduction to a one-dimensional model
becomes substantially more challenging.

\subsection{From an example to a mechanism}

Our starting point is the fractal dimension theory of the geometric Lorenz
attractor, one of the most celebrated examples of a chaotic flow that is
not uniformly hyperbolic. It was introduced independently in \cite{ABS,G} as a
model for the strange attractor found by Lorenz in his study of atmospheric
convection. These are three-dimensional flows possessing an attracting set that
contains a hyperbolic singularity accumulated by regular orbits; they are the
prototype of the class of \emph{singular-hyperbolic} attractors which, by
\cite{MPP04}, is exactly the class of robustly transitive attractors of
three-dimensional flows. It was shown in \cite{AAPP} that $C^{1+\alpha}$
singular-hyperbolic attractors have zero volume. 
\color{black}

Estimates from below for the dimension of the Lorenz attractor go back to
\cite{AP87, Ste00}, and numerical experiments give $\dim_H\approx2.06$ for the
classical Lorenz equations, with a small margin of error \cite{Vi04}.
A complete proof for the geometric model was obtained only much later in
\cite{MPR}, where Moreira, Pac\'ifico and Roma\~na showed that a geometric Lorenz
attractor contains \emph{hyperbolic horseshoes} of Hausdorff dimension strictly
greater than $2$. This statement is both stronger and dynamically more useful than a lower bound for $\dim_H(\Lambda)$
itself, as evidenced by their deduction that for a large set of observables the associated Markov and Lagrange dynamical spectra have persistently non-empty interior. This
is, to date, the only class of singular flows for which such a statement is known, leaving a large set of important examples unexplored, including the wider class of singular-hyperbolic attractors.

The proof in \cite{MPR} proceeds by reduction. A geometric Lorenz attractor
is given as the suspension flow of a skew product map 
preserving a contracting foliation, and a one-dimensional quotient map $f:I\to
I$; the heart of the matter is to construct, inside $I$, regular Cantor sets of
Hausdorff dimension arbitrarily close to $1$, which are then lifted to the
section and suspended. That construction is explicit and depends heavily on the
fact that the Lorenz map is \emph{uniformly expanding}, $|f'|>\sqrt2$. This is
precisely where the argument stops. As soon as the one-dimensional model loses
uniform expansion, the construction of \cite{MPR} has no analogue. On the other
hand, several classes of non-uniformly hyperbolic attractors lack this uniform
expansion. One emblematic case is the \emph{contracting Lorenz}, or
\emph{Rovella}, attractor of \cite{Ro93}. It is obtained from the geometric
Lorenz model by changing the eigenvalues at the singularity, and its quotient
map has unbounded derivative at the singularity but no uniform expansion, so
that \cite{MPR} says nothing about it. Since the Rovella attractor is
persistent but not robust, and is frequently regarded as the singular-flow
analogue of the quadratic family, it is arguably the first test case any such
theory should pass. The same obstruction affects every attractor whose quotient
map has critical points.

The purpose of this paper is twofold. We first bridge this fractal theory from
the expanding to the non-expanding world by identifying a more general mechanism. Uniform expansion is replaced by the existence of a full-branch
Gibbs--Markov \emph{induced} map, a structure that is available in essentially
every known class of non-uniformly expanding one-dimensional model which
yields analogous results for a much wider class of attractors. Secondly, we
explore the flexibility of the Hausdorff dimension spectrum for non-hyperbolic
attractors by showing that in many cases one can realize arbitrary intermediate
stable, unstable and metric dimensions. We stress that this second group of
results is new even for the geometric Lorenz attractor itself.

\subsection{Hausdorff dimension of non-uniformly hyperbolic attracting sets}

We now present our main results. Given the technical character of the
statements, we introduce here only the notions that are strictly needed, and
postpone the more technical and precise definitions to Section
\ref{sec:prelim}. Throughout the text, $M$ denotes a compact three-dimensional
manifold, $X$ denotes a $C^{1+\alpha}$ vector field on $M$, and $(X_t)_t$
denotes the flow induced by $X$. The orbit, positive orbit and negative orbit of
a point $x \in M$ are defined, respectively, as the sets
$$
\mathcal{O}(x)=\{X_t(x): t \in \mathbb{R}\},\quad
\mathcal{O}^+(x)=\{X_t(x): t \geq 0\},\quad
\text{and}\quad
\mathcal{O}^-(x)=\{X_t(x): t \leq 0\}.
$$
A compact set $\Lambda \subset M$ is called \emph{$X$-invariant} if
$X_t(\Lambda)=\Lambda$ for every $t \in \mathbb{R}$. An $X$-invariant set
$\Lambda \subset M$ is called an \emph{attracting set} if there exists an open
neighborhood $U$ of $\Lambda$ such that
$$
X_t(\overline{U}) \subset U \quad \text{for every } t>0,
$$
and
$$
\Lambda = \bigcap_{t \geq 0} X_t(U).
$$
In this case the set $U$ is called a \emph{trapping region} for $\Lambda$. We
say that $\Lambda$ is \emph{transitive} if there exists $x \in \Lambda$ such
that $\overline{\mathcal{O}^+(x)}=\Lambda$. A compact invariant set is called an
\emph{attractor} for $X$ if it is a transitive attracting set.

A central ingredient in our approach is a sequence of reductions of the original
dynamics to lower-dimensional maps, available for many classes of flows (see
Section \ref{sec:prelim} for a detailed explanation of this reduction process).
Starting from an attracting set $\Lambda$, in many cases it is possible to
describe its dynamics through a Poincar\'e return map $\mathcal R:S \to S$
defined on a suitable global cross-section $S$. This provides a, possibly
discontinuous, discrete-time model that faithfully encodes the dynamics of
$\Lambda$. Under additional assumptions the return map preserves the stable
foliation and therefore induces a one-dimensional quotient map $f:I\to I$,
obtained by identifying points lying on the same stable leaf. The dynamics of
these three objects are closely related, which allows dynamical information to
be transferred between them.

In the case of the Lorenz attractor the one-dimensional model $f$ is uniformly
expanding, which is exactly the setting supporting the methods of \cite{MPR},
but in more general settings uniform expansion cannot be granted. Nevertheless,
in many cases one can recover strong properties for the so-called induced
dynamics, whose definition we now recall. Let $J\subset I$ be a subinterval and
let $\mathcal P$ be a countable measurable partition of $J$ such that
$$\Leb\left(J\setminus \bigcup_{P\in \mathcal{P}}P\right)=0.$$ We say that $F:J\to J$ is an
induced map of $f$ if there exists a return time function $R:J\to\mathbb N$ which is constant at each element of the partition and satisfies
that
\begin{equation}
    \label{eq:definduc}
F(y)=f^{R(y)}(y), \quad \text{for every $y\in J$}.
\end{equation}

The countable measurable partition ${\mathcal P}$ of $J$ is said to be
\emph{Markov} for $F$ if, for every $P\in{\mathcal P}$, the restriction
$$
F|_P:P\to F(P)
$$
is a measurable bijection. The map $F$ is called a \emph{Gibbs--Markov map} if
there exist a metric $d$ on $J$, constants $C>0$ and $\theta\in(0,1)$, and a
Markov partition ${\mathcal P}$ such that the following hold:
\begin{enumerate}
    \item there exists $c_0>0$ such that
$\Leb(F(P))\geq c_0$
for every  $P\in{\mathcal P}$,
\item $d(Fx,Fy)\geq \theta^{-1}d(x,y)$ for every $x,y\in P$, $P\in{\mathcal P}$,
\item
if $g_P$ denotes the inverse Jacobian of $F|_P$, then
$|\log g_P(x)-\log g_P(y)|
\leq
C\, d(Fx,Fy)$
for every $x,y\in P$.
\end{enumerate}
If, in addition,
\begin{enumerate}
    \item[(4)] $F(P)=J$ for every $P\in \mathcal P$,
\end{enumerate}
then we say that $F$ is a \emph{full-branch Gibbs--Markov map}.

\medskip
Observe that once an induced map $F$ has been obtained, one can repeat the
process and consider a map induced by $F$, and this procedure may be iterated
finitely many times. In non-uniformly expanding dynamics it often happens that
Gibbs--Markov maps arise as induced maps, and that they can be made expanding
(see e.g. \cite{Alvesbook,You98} and
references therein).

We are now in a position to define the main object of our study.
\begin{definition}\label{def:modeled}
A  partially hyperbolic  attracting set $\Lambda\subset M$ 
admitting a $C^1$-stable foliation
is said to be \emph{modeled by a
full-branch expanding Gibbs--Markov map} if it admits a one-dimensional model
map which admits a full-branch, expanding, 
Gibbs--Markov transformation.
\end{definition}

Our first result concerns the existence of horseshoes with large dimension for
the class of attractors just described.

\begin{maintheorem}\label{thm:dimension-flows}
\emph{
Let $M$ be a three-dimensional compact manifold and let $X\in
\mathfrak{X}^{1+\alpha}(M)$ have an attracting set  $\Lambda\subset M$ modeled by a full-branch expanding Gibbs--Markov map. Then, there
exists a hyperbolic horseshoe $K\subset \Lambda$ such that
$
\dim_H(K) > 2.
$
}
\end{maintheorem}

Since the quotient map of a geometric Lorenz attractor is uniformly expanding with a singularity and can be modeled by a full branch, expanding Gibbs-Markov map,
Theorem~\ref{thm:dimension-flows} contains the corresponding result of
\cite{MPR} as a particular case. Let us indicate where the difficulty lies
(cf. Section~\ref{sec:applications} for more details). 
 Two points therefore require care. First, the Cantor sets of dimension
close to $1$ must be produced inside the induced dynamics and then pushed back
to the original model, which forces one to control the distortion of $f^{R}$ on
elements of the Markov partition and to select branches with a uniform bound on
$R$. Second, when lifting to the cross-section and suspending, the roof function
is unbounded, so the dimension of the resulting invariant set in $M$ has to be
compared with that of its section through an argument that does not presuppose a
bounded return time. Both points are carried out in Section \ref{sec:ThmA}.

\begin{remark}\label{rmk:PalisViana}
The continuity of the Hausdorff dimension of hyperbolic horseshoes under $C^1$ perturbations of surface diffeomorphisms was established by Palis and Viana \cite{PV}: if $(\varphi_t)_{t\in I}$ is a $C^1$ family of surface diffeomorphisms and $(\widehat K_t)_{t\in I}$ is the corresponding continuation of a hyperbolic horseshoe, then the map
$
t \mapsto \dim_H(\widehat K_t)
$
is continuous. In particular, using the continuity of the Hausdorff dimension of horseshoes for Poincar\'e maps, as a consequence of Theorem~\ref{thm:dimension-flows} one concludes the robust existence of horseshoes with Hausdorff dimension strictly larger than $2$.
\end{remark}

\subsection{Lagrange and Markov dynamical spectra}

Following the program of \cite{MPR}, we can also extend their results to the
setting of Lagrange and Markov dynamical spectra. Let us briefly recall these
concepts. For a given irrational number $\alpha$, the classical \emph{Lagrange
value} of $\alpha$ measures how well $\alpha$ can be approximated by rationals,
and the \emph{Lagrange spectrum} is the set $L$ of finite values so obtained as
$\alpha$ ranges over the irrational numbers. The \emph{Markov spectrum} $M$ is
defined analogously, replacing a $\limsup$ by a supremum. Both sets are closed,
$L\subset M$, both are intricate near their lower ends, and both contain a
half-line. The reason these two sets belong to dynamics rather than to
arithmetic is the continued fraction algorithm: it converts $\ell(\alpha)$ into
the asymptotic maximum of a continuous function along the orbit of $\alpha$
under a hyperbolic map, so that $L$ and $M$ are the sets of asymptotic and of
total maxima of an observable along orbits of a dynamical system. We refer to
\cite{CF89,LMMR20} for the classical theory.

This dynamical reading suggests an immediate generalization. Given a flow
$(X_t)_t$ on a manifold $M$, an invariant set $\Lambda\subset M$ and a
continuous observable $\phi:M\to\R$, the associated \emph{Markov} and
\emph{Lagrange dynamical spectra} are
$$
\Mspec\big((X_t)_t,\Lambda,\phi\big)=\Big\{\sup_{t\in\R}\phi(X_t(x)):x\in\Lambda\Big\},
\qquad
\Lspec\big((X_t)_t,\Lambda,\phi\big)=\Big\{\limsup_{t\to+\infty}\phi(X_t(x)):x\in\Lambda\Big\}.
$$
The basic qualitative question is whether these sets are large --- for instance,
whether they have non-empty interior. A remarkable series of works of Moreira
and Roma\~na, and of Moreira, Pac\'ifico and Roma\~na, has shown that largeness
of the spectra is governed by a single fractal invariant: the Hausdorff
dimension of the invariant set. For horseshoes of surface diffeomorphisms,
non-empty interior holds for typical observables as soon as the horseshoe has
Hausdorff dimension greater than $1$ \cite{MR17}; for geodesic flows and for
generic Anosov flows analogous statements are available \cite{MR23}; and the
finer structure of the spectra has since been studied in \cite{CMM18,LM20}.

\begin{maintheorem}\label{cor:Markov-spectrum}
 \emph{Let $M$ be a three-dimensional closed manifold and let $X\in
\mathfrak{X}^{2}(M)$ be a $C^{2}$ vector field and 
$\Lambda\subset M$ be an attracting set with hyperbolic singularities. Let $U$ be an attracting neighborhood of $\Lambda$ and
assume that the flow on $\Lambda$ is modeled by a full-branch expanding
Gibbs--Markov map. For every $\vep>0$ there are $Y\in
\mathfrak{X}^2(M)$ that is $\vep$-close to $X$ in the $C^2$-topology  and an open
neighborhood  $\mathcal{U}\subset \mathfrak{X}^2(M)$ of $Y$} such that
for every $\widetilde Y\in \mathcal{U}$ there is an open and dense subset
$\mathcal{H}_{\widetilde Y}\subset C^1(U,\R)$ with the property that, for every $\phi\in
\mathcal{H}_{\widetilde Y}$,
$$\operatorname{int}\big(\Lspec(({\widetilde Y}_t)_t,\Lambda_{\widetilde Y},\phi)\big)\neq \emptyset \quad \textrm{ and }\quad  \operatorname{int}\big(\Mspec((\widetilde Y_t)_t,\Lambda_{\widetilde Y},\phi)\big)\neq \emptyset.$$

\end{maintheorem}

\medskip

\subsection{Flexibility of Hausdorff dimension}

An interesting question that has attracted significant attention within the
dynamical systems community is the following: given a dynamical system and a
dynamical quantity, under what conditions does this quantity achieve a
continuous spectrum across its subsystems? A classical benchmark in this
direction is the so-called intermediate entropy property investigated by Katok
in \cite{Ka80}. A dynamical system is said to possess the intermediate entropy
property if every value between $0$ and its topological entropy is realized as
the metric entropy of one of its ergodic measures. While the entropy spectrum of
certain systems is highly rigid (e.g. uniquely ergodic systems with positive
entropy; see \cite{G72, BCR07}), many others display full flexibility (e.g.
$C^2$ surface diffeomorphisms \cite{Ka80}). Besides metric entropy, several
other dynamical quantities, such as topological entropy, topological pressure,
Lyapunov exponents and Hausdorff dimension, have been investigated in order to
establish intermediate behaviour in various contexts 
(see e.g.
\cite{AOR25, DGRZ25, DHT22, HXZ10, LSWW20, S25, TWW19}).

A natural question is whether a non-uniformly hyperbolic structure provides a
sufficiently rich geometric setting for the Hausdorff dimension of invariant
sets to satisfy an intermediate dimension property. Let us be more precise.
Denote
$$
\Hor(\Lambda)=\Big\{K\subset \Lambda \;\colon\;  K \textrm{ is a horseshoe}\Big\} \quad \text{and} \quad
\gamma=\sup\Big\{\dim_H(\mu) \;\colon\; K\in \Hor(\Lambda), \;\mu\in \Mer{X}{K} \Big\}.
$$
Here the situation is considerably more complicated than its entropy
counterpart. Indeed, Hausdorff dimension is a fragile and irregular quantity,
requiring strong geometric structure in order to be reasonably well behaved.
A word of caution is in order concerning the upper endpoint $\gamma$.
McCluskey and Manning \cite[Theorem~2]{MM83} proved that typical horseshoes in
dimension $2$ have Hausdorff dimension strictly larger than the supremum of the
dimensions of their invariant ergodic measures. Consequently, even though in
many non-uniformly hyperbolic attractors one can approximate the entropy and the
exponents of the attractor by those of hyperbolic horseshoes, it may still
happen that $\gamma\neq \dim_H(\Lambda)$. Fortunately, one can nevertheless
achieve every intermediate metric dimension in the interval $[1,\gamma]$.

\begin{maintheorem}\label{thm:intermediate-metric-dimensions}
\emph{ 
   Let $M$ be a three-dimensional compact manifold and let $X\in \mathfrak{X}^{1+\alpha}(M)$ be a $C^{1+\alpha}$ vector field with an attractor  $\Lambda\subset M$. If
$\Lambda$ is modeled by a full-branch expanding Gibbs--Markov map, then for every $c\in [1,\gamma)$ there is an ergodic measure $\mu_c\in \Mer{X}{\Lambda}$ such that $\dim_H(\mu_c)=c.$
}
 \end{maintheorem}

Recently, a program towards a topological version of Katok's intermediate
entropy property was initiated in \cite{AOR25}. This version asks whether every
intermediate value of entropy can be realized as the topological entropy of a
subsystem, i.e. of a compact invariant set. The main difficulty in this approach
lies in the need for fine control over the variational principle applied to the
support of a given ergodic measure realizing the intermediate value. Indeed,
taking the supremum of the entropies along the support may cause the topological
entropy to explode. The situation is even more delicate for flows, since
horseshoes are given as suspensions of subshifts and hence their entropy depends
strongly on the integral of the roof function, by Abramov's formula. To overcome
this difficulty, \cite{AOR25} adopted the strategy of constructing strictly
ergodic subsystems with fine control over both the entropy and the roof
function.

Here the situation is more challenging. As discussed above, there is no
guarantee of the existence of a variational principle for Hausdorff dimension.
Moreover, the Hausdorff dimension of a measure $\mu$ supported on a hyperbolic
horseshoe decomposes as

$$
\dim_H(\mu)=\dim_H^s(\mu)+\dim_H^u(\mu)-1,
$$
where $\dim_H^s(\mu)$ and $\dim_H^u(\mu)$ are the stable and unstable dimensions
of $\mu$, respectively, given by ~\eqref{eq:def-dim-su}. 
 Consider the quantities
$$
\gamma^s=\sup\Big\{\dims(\mu) \;\colon\; K\in \Hor(\Lambda),\; \mu\in \Mer{X}{K} \Big\}$$
and 
$$
\qquad
\gamma^u=\sup\Big\{\dimu(\mu) \;\colon\; K\in \Hor(\Lambda), \;\mu\in \Mer{X}{K} \Big\}.
$$

The strategy of constructing strictly ergodic subsystems is very effective even
in the absence of a variational principle and, by applying coding techniques to
the geometric stable and unstable potentials, we obtain compact invariant sets
with any prescribed intermediate stable and unstable Hausdorff dimensions.

\begin{maintheorem}\label{thm:intermediate-dimensions}
\emph{ 
   Let $M$ be a three-dimensional compact manifold and let $X\in \mathfrak{X}^{1+\alpha}(M)$ be a $C^{1+\alpha}$ vector field with an attractor  $\Lambda\subset M$. Suppose that $(X_t)_t$ on $\Lambda$ is modeled by a suspension flow over a $C^{1+\alpha}$ full-branch expanding Gibbs--Markov map of a compact interval $I\subset\mathbb R$. Then for every $c\in [1,\gamma^s)$ and every $c'\in [1,\gamma^u)$ there are strictly ergodic hyperbolic sets $K_c,K_{c'}\subset \Lambda$ such that
     $$\dims(K_{c})=c \quad \text{and} \quad  \dimu(K_{c'})=c'.$$
     }
    \end{maintheorem}

Together with Theorem~\ref{thm:dimension-flows}, the two previous results
guarantee the existence of hyperbolic sets and measures whose interplay gives
rise to a rich multifractal structure from the point of view of Hausdorff
dimension. Let us now make a few remarks concerning these statements. First, in
contrast with the analogous results for intermediate entropy, the intervals
covered here are bounded below by $1$ rather than by $0$; this is due to the
fact that the smallest possible dimension of a set containing the orbit of a
regular point of a flow equals $1$. Second, while
Theorem~\ref{thm:intermediate-dimensions} provides compact invariant subsets
with prescribed (un)stable dimensions,
Theorem~\ref{thm:intermediate-metric-dimensions} concerns the dimension of
ergodic measures.

Finally, one may ask for a strictly ergodic version of
Theorem~\ref{thm:intermediate-metric-dimensions}, obtained by prescribing the
stable and the unstable dimensions simultaneously and invoking the
Ledrappier--Young formula. We do not expect this to be possible, and the
reason is structural rather than technical: for a given pair $c,c'$ as in
Theorem~\ref{thm:intermediate-dimensions}, the sets $K_c$ and $K_{c'}$ produced
by our construction need not coincide and, even when they do, the equilibrium
states associated with the stable and with the unstable geometric potentials are
in general distinct, which is incompatible with strict ergodicity. In this sense
Theorem~\ref{thm:intermediate-dimensions} appears to be optimal.

\subsection{Applications}

The hypotheses of Definition~\ref{def:modeled} are verified in
Section~\ref{sec:applications} for several classes. In each case the
verification is modular: one exhibits the cross-section and the quotient map,
and checks that the induced map is a full-branch expanding Gibbs--Markov map.
Consequently all these examples satisfy Theorem~\ref{thm:dimension-flows} and
Theorem~\ref{cor:Markov-spectrum}, and, since in each case the model is
in fact a suspension over such a map, Theorems
\ref{thm:intermediate-metric-dimensions} and
\ref{thm:intermediate-dimensions} as well. Since each class has its own
particularities, in some cases the results can be further improved, and hence
each case is treated separately. The next table lists the classes to which our
results apply.

\medskip
\begin{center}
\begin{tabularx}{\linewidth}{p{5.4cm} p{4.2cm} X}
\hline
\textbf{Class} & \textbf{Previously known?} & \textbf{Here}\\
\hline
Geometric Lorenz attractors & \cite{MPR} & recovered\\
Singular-hyperbolic attracting sets & No & Theorem~\ref{thm:app-sing-hyp}\\
Robustly transitive attractors & No & Corollary~\ref{cor:robust-transitive}\\
Rovella (contracting Lorenz) attractors & No & Theorem~\ref{thm:app-rovella}\\
General non-uniformly hyperbolic skew product suspensions & No & Theorem~\ref{thm:skew-product-suspension}\\
\hline
\end{tabularx}
\end{center}
Among these, some  deserve to be singled out. The Rovella attractor is the persistent, non-robust counterpart of the geometric Lorenz model, often regarded as the analogue of the quadratic family for singular flows. Although structurally similar to the geometric Lorenz attractors —suggesting that the techniques in \cite{MPR} might apply directly — its quotient map fails to be uniformly expanding, preventing their direct application. General skew product suspensions cover attracting sets obtained by suspending a skew-product map with a contracting foliation and a one-dimensional non-uniformly hyperbolic map. This suspension process matches the one used in \cite{MPR} to analyze the geometric Lorenz attractor model; the difference here lies in choosing a one-dimensional map from a broader class that fails to satisfy the uniform expansion property displayed in the case of the Lorenz maps. Consequently, our methods can be extended to various other models, provided the induced maps satisfy the main Gibbs–Markov assumption. To the best of our knowledge, Theorems~\ref{thm:app-rovella} and \ref{thm:skew-product-suspension} provide the first proofs that these systems contain horseshoes with a Hausdorff dimension greater than two, as well as the first results regarding their Lagrange and Markov spectra.

\subsection{Organization of the paper}

Section~\ref{sec:prelim} fixes the setting and collects classical results that will
be used through the text.
Section~\ref{sec:ThmA} proves Theorem~\ref{thm:dimension-flows}.
Section~\ref{sec:markov} proves Theorem~\ref{cor:Markov-spectrum}.
Section~\ref{sec:interdiate} proves
Theorems~\ref{thm:intermediate-metric-dimensions} and
\ref{thm:intermediate-dimensions}. Section~\ref{sec:applications} verifies our results for the classes listed in the table above. In the last section we provide the reader with a table describing the notation used in this text.

\section{Preliminaries}
\label{sec:prelim}

In this section we fix the notation and collect the definitions and the known
results that will be used throughout the paper. Unless stated otherwise, $M$
denotes a three-dimensional compact manifold and $X$ a $C^{1+\alpha}$ vector field on $M$, whose
induced flow is denoted by $(X_t)_{t\in\R}$. We write $\mathfrak{X}^{r}(M)$ for the set
of $C^{r}$ vector fields on $M$ endowed with the $C^{r}$ topology, and
$\Sing(X)$ for the set of singularities of $X$, i.e., $\sigma\in Sing(X)$ if and only if $X(\sigma)=0$. A point which is not a singularity is called \emph{regular}.

Recall that a probability measure $\mu$ is $X$-invariant if $\mu(X_t(A))=\mu(A)$ for every
$t\in\R$ and every Borel set $A$, and that $\mu$ is ergodic if
$\mu(A)\mu(A^{c})=0$ for every invariant Borel set $A$. The sets of $X$-invariant Borel probability measures and ergodic measures are, respectively, denoted by $\Msr{X}{M}$ and $\Mer{X}{M}$. If $\Lambda$
is compact and invariant set, we denote by $\Msr{X}{\Lambda}$ and $\Mer{X}{\Lambda}$ for the sets of invariant and ergodic measures of the flow, when restricted to $\Lambda$.  $\Lambda$ is \emph{uniquely ergodic} if 
 $ \#\Msr{X}{\Lambda}=1$. A compact invariant set $\Lambda$  is \emph{minimal} if it contains no non-trivial, proper, closed invariant subset  and it is
\emph{strictly ergodic} if it is both minimal and uniquely ergodic. 
A compact invariant set $\Lambda$ is \emph{locally maximal} if there is an open
neighbourhood $V\supset\Lambda$ such that $\Lambda=\bigcap_{t\in\R}X_t(V)$.

\subsection{Hyperbolicity, dominated splittings and horseshoes}
\label{ss:hyperbolicity}

A compact invariant set $\Lambda\subset M$ is \emph{hyperbolic} for $X$ if
there exist constants $C>0$, $\lambda>0$, and a continuous $DX_t$-invariant
splitting
$$
T_\Lambda M = E^s \oplus \langle X(x)\rangle \oplus E^u,
$$
such that, for all $x\in\Lambda$ and $t>0$,
$$
\|DX_t(x)v\| \le C e^{-\lambda t} \|v\|, \quad \forall v\in E^s(x)
\qquad\textrm{and}\qquad
\|DX_{-t}(x)v\| \le Ce^{-\lambda t} \|v\|, \quad \forall v\in E^u(x).
$$

We say that $\sigma\in\Sing(X)$ is a \emph{hyperbolic singularity} provided
$\{\sigma\}$ is a hyperbolic set. For any $\sigma\in\Sing(X)$ we define its
stable and unstable sets, respectively, as
$$
W^s(\sigma)=\{x\in M; d(X_t(\sigma),X_t(x))\to 0,\ t\to +\infty\}
$$
and
$$
W^u(\sigma)=\{x\in M; d(X_t(\sigma),X_t(x))\to 0,\ t\to -\infty\}.
$$
When $\sigma$ is hyperbolic, $W^s(\sigma)$ and $W^u(\sigma)$ are $C^1$-immersed
submanifolds. We now isolate the classes of hyperbolic sets on which the results of this
paper are built.
 
 A \emph{basic set} is a compact, invariant, transitive,
locally maximal hyperbolic set. A basic set of a surface diffeomorphism is
\emph{of saddle type} if $\dim E^{s}=\dim E^{u}=1$ on it. A \emph{horseshoe} for the flow $(X_t)_t$ is a compact, invariant, transitive,
locally maximal hyperbolic set which is not a single periodic orbit.

\begin{remark} 
Locally maximal hyperbolic sets persist under small perturbations, together
with their dynamics; this is the content of the following classical result,
which will allow us to follow the objects we construct along a $C^1$-family of
vector fields. For flows, horseshoes are given as the suspension flows of horseshoe diffeomorphism. 
\end{remark}

The set $\Lambda\subset M$ admits a \emph{dominated splitting}
if there are constants $C,\lambda>0$ and a non-trivial $DX_t$-invariant
decomposition $T_\Lambda M=E\oplus F$ such that
$$
\frac{\|DX_t(x)|_{E_x}\|}{m\big(DX_t(x)|_{F_x}\big)}\le C e^{-\lambda t},
\qquad x\in\Lambda,\ t\ge0,
$$
where $m(A)=\|A^{-1}\|^{-1}$ denotes the minimum norm of a linear isomorphism
$A$. We say that $\Lambda$ is \emph{partially hyperbolic} if it admits a
dominated splitting $T_\Lambda M=E^{s}\oplus E^{cu}$ whose first bundle is
uniformly contracted:
\begin{equation}\label{eq:contraction}
\|DX_t|_{E^{s}_x}\|\le Ce^{-\lambda t},\qquad x\in\Lambda,\ t\ge0 .
\end{equation}

\subsection{Cross-sections, Poincar\'e maps and suspension flows}
\label{ss:sections}

The strategy adopted throughout the paper is to reduce the flow to a surface
diffeomorphism, when a contracting direction is available, an then further reduce it to a
one-dimensional map. We describe here the two reductions.

\subsubsection*{Suspension semiflows}

Given a map $G:\Delta\to\Delta$ and a roof function $\tau:\Delta\to\mathbb R^+$,
the \emph{suspension semiflow} $(G_t)_{t\geq0}$ is defined by the vertical
displacement
$$
G_t(y,u)=(y,u+t)
$$
modulo the identification of the suspension space $\Delta^\tau$, defined by
$$
\Delta^\tau
=
\{(y,u)\in \Delta\times\mathbb R:0\leq u\leq \tau(y)\}/\sim
\quad \text{where \quad $(y,\tau(y)) \sim(Gy,0).$}
$$
If $G$ is invertible then $(G_t)_t$ defines a flow, and in this case the fiber
$\Delta\times\{0\}$ is trivially a cross-section. The converse construction is
carried out below.

\subsubsection*{Cross-sections and Poincar\'e maps}

We say that a surface $S\subset M$ is a \emph{local cross-section} for $X$ if,
for every $x\in S$, the transversality condition
$$
T_x S \oplus \langle X(x) \rangle = T_x M, \textrm{ for all } x\in S,
$$
holds. Furthermore, we require that the angle between $T_x S$ and the line
field $\langle X(x)\rangle$ be uniformly bounded away from $0$ on compact
subsets of $S$. Observe that, by definition, every point in a local
cross-section is a regular point of $X$.

We say that a compact invariant set $\Lambda\subset M$ admits a \emph{global
cross-section} if there exist finitely many pairwise disjoint local
cross-sections $S_1,\dots,S_k$ such that the (possibly disconnected) surface
$S=S_1\cup\cdots\cup S_k$ satisfies
$$
S \cap \mathcal{O}^+(x) \neq \varnothing, \quad \forall x \in \Lambda \setminus W^s(\Sing(X)),
$$
and
$$
S \cap \mathcal{O}^-(x) \neq \varnothing, \quad \forall x \in \Lambda \setminus W^u(\Sing(X)).
$$
In the case $\Sing(X)=\emptyset$, these conditions reduce to
$$ S \cap \mathcal{O}^+(x)\neq\emptyset \textrm{ and } S \cap \mathcal{O}^-(x)\neq\emptyset,$$
for every $x\in\Lambda$.

If $\Lambda$ admits a global cross-section, one can define the Poincar\'e map
$\mathcal R$ on $S\cap\Lambda$ as follows. For any
$$
x \in S \cap \bigl(\Lambda \setminus W^s(\Sing(X))\bigr),
$$
let $\tau(x)>0$ be the smallest positive time such that $X_{\tau(x)}(x)\in S$,
and set
$$
\mathcal R(x) = X_{\tau(x)}(x).
$$
The inverse map is defined analogously by considering the backward orbit: for
any
$$
y \in S \cap \bigl(\Lambda \setminus W^u(\Sing(X))\bigr),
$$
let $\tau^-(y)>0$ be the smallest positive time such that $X_{-\tau^-(y)}(y)\in
S$ (that is, the first return to $S$ in negative time), and define
$$
\mathcal R^{-1}(y) = X_{-\tau^-(y)}(y).
$$
Consequently, $\mathcal R$ is a bijection from
$$
S \cap \bigl(\Lambda \setminus (W^s(\Sing(X)) \cup W^u(\Sing(X)))\bigr)
$$
onto itself. It is well known that $\mathcal R$ is continuous on
$S\cap(\Lambda\setminus W^s(\Sing(X)))$, while $\mathcal R^{-1}$ is continuous
on $S\cap(\Lambda\setminus W^u(\Sing(X)))$. The return time, however, is
unbounded near the singularities: $\tau(x)\to+\infty$ as $x$ approaches
$W^s(\Sing(X))$, and $\tau^-(y)\to+\infty$ as $y$ approaches $W^u(\Sing(X))$.

When $\Sing(X)\cap\Lambda$ is a hyperbolic set, one can always construct a
global cross-section and a Poincar\'e map for $\Lambda$, and this construction
extends to any $C^1$-close vector field $Y$ and to the continuation $\Lambda_Y$
of $\Lambda$ given. A more detailed
explanation will be provided in Section~\ref{sec:markov}; for a complete and
precise construction we refer the reader to \cite{AM19}.

The two constructions above are inverse to each other in the following sense:
if an attracting set $\Lambda$ admits a global cross-section $S$ and a
Poincar\'e map $\mathcal R:S\to S$, then the dynamics of $X$ restricted to
$U\setminus(W^s(\Sing(X))\cup W^u(\Sing(X)))$ is the suspension flow of
$\mathcal R$ with roof function given by the first return time to $S$.

\subsubsection*{Hyperbolicity of the Poincar\'e map}

Since the flow is $C^1$ and the cross-sections are transverse to $X$, the map
$\mathcal R$ is a $C^1$-diffeomorphism on a neighbourhood $V$ of $\Lambda$
excluding the stable and unstable sets of the singularities. Moreover,
hyperbolicity is inherited by the section: if $K\subset\Lambda$ is a hyperbolic
set for $X$, then there is a neighbourhood $V$ of $\widehat K=K\cap S$ such
that $\widehat K$ is a hyperbolic set of the diffeomorphism
$\mathcal R:V\to\mathcal R(V)$. Precisely, the hyperbolic splitting of $K$
induces a continuous $D\mathcal R$-invariant splitting
$$
T_{\widehat K} S = E^s \oplus E^u
$$
and constants $C>0$, $\lambda>0$ such that, for all $x\in\widehat K$ and
$n\ge1$,
$$
\|D\mathcal R^n(x)v\| \le C e^{-\lambda n} \|v\|, \quad \forall v\in E^s(x),
\qquad\textrm{and}\qquad
\|D\mathcal R^{-n}(x)v\| \le C e^{-\lambda n} \|v\|, \quad \forall v\in E^u(x).
$$

\begin{notation}\label{not:pushpull-horseshoes}
Throughout, the horseshoes of the Poincar\'e map that we consider are  denoted
$\widehat K$; those of the flow are denoted $K$, so that
$\widehat K=K\cap S$.
\end{notation}

\subsubsection*{The one-dimensional quotient}

Assume now that $\Lambda$ is a partially hyperbolic set for $X$ with invariant
splitting
$$
T_{\Lambda}M = E^s \oplus E^{cu},
$$
where $\dim E^s=1$ and $\dim E^{cu}=2$. Since partial hyperbolicity is an open
condition, this splitting admits a continuous extension to a small tubular
neighborhood $U$ of $\Lambda$.

By the stable manifold theorem \cite{HPS}, the uniformly contracting bundle
$E^s$ integrates into an $X$-invariant strong stable H\"older continuous foliation
$\mathcal F^{ss}$, whose leaves are one-dimensional curves tangent to  $E^s$. Higher regularity requires additional
bunching assumptions. For instance, under the bunching condition
$$
\sup_{x\in\Lambda}
\left\|DX^{t_0}_x|_{E^s}\right\|
\left\|DX^{t_0}_x\right\|^2<1
$$
for some $t_0>0$, the stable foliation is of class $C^2$
(see e.g. \cite{AraButVar,AM17}).
For any $x\in U$, define
$$\mathcal{F}^s(x)=\bigcup_{t\in \R} X_t(\mathcal{F}^{ss}(x)).$$
Restricting $\mathcal{F}^s$ to the two-dimensional global cross-section $S$
constructed above, we obtain a one-dimensional foliation of $S\cap U$.

Suppose that each connected component $U_i$ of $U$ has a product structure
$L\times I_i$, in such a way that $I_i$ is the leaf space of
$\mathcal{F}^s|_{S\cap U_i}$, i.e. the quotient space obtained by identifying
points of $S\cap U$ lying on the same stable leaf, and let $\pi^s$ denote the
corresponding quotient projection (see \cite{AM19} for the precise details). Because the flow $X_t$ preserves the stable
foliation $\mathcal{F}^s$, the Poincar\'e map $\mathcal R$ maps stable leaves
in $S$ to stable leaves in $S$. Hence $\mathcal R$ induces a well-defined
interval map
$$
f: I \to I, \qquad f(\pi^s(x)) = \pi^s(\mathcal R(x)),
$$
where $I=I_1\cup\cdots\cup I_k$.

The projected map $f$ need not be continuous: it may exhibit discontinuities
precisely at the images under $\pi^s$ of the stable manifolds of the
singularities, where the forward orbit fails to return to the section in finite
time. As we will see later, such an induction is available for several classes
of attractors, and the construction of $f$ will be explained in each class as
it appears in the applications. From now on we assume that every attractor
under consideration admits a Poincar\'e map which can be quotiented into a
one-dimensional map $f$.

\subsection{Lyapunov exponents and hyperbolic measures}
\label{ss:lyapunov}

Let $\mu$ be an invariant probability measure for $X$. Oseledets' theorem
\cite{O68} asserts that, for $\mu$-almost every $x\in M$, there exist numbers
$\lambda_1(x) < \lambda_2(x) < \cdots < \lambda_{k(x)}(x)$
and a $DX_t$-invariant splitting of the tangent space
\[
T_x M = \widehat E_1(x) \oplus \widehat E_2(x) \oplus \cdots \oplus \widehat E_{k(x)}(x)
\]
such that, for every $v\in \widehat E_i(x)\setminus\{0\}$,
\[
\lim_{t\to\pm\infty} \frac1t \log \|DX_t(x)v\| = \lambda_i(x).
\]
The numbers $\lambda_i(x)$ are the \emph{Lyapunov exponents} of $\mu$ at $x$,
and the subspaces $\widehat E_i(x)$ are the corresponding \emph{Oseledets
spaces}. The value $k(x)\le\dim M$ is the number of distinct Lyapunov exponents
at $x$, and the functions $x\mapsto\lambda_i(x)$ and
$x\mapsto\dim\widehat E_i(x)$ are measurable and invariant along orbits. Since
the flow direction is always neutral, one of the exponents is zero; its
corresponding Oseledets space contains the flow direction $\langle X(x)\rangle$
and will be denoted by $\widehat E^0(x)$.

We say that an invariant measure $\mu$ is a \emph{hyperbolic measure} if its
only vanishing Lyapunov exponent is the one associated with the flow direction.
For an ergodic measure $\mu$ the exponents and their multiplicities are
constant $\mu$-almost everywhere, and we write them as
\[
\lambda_1 < \lambda_2 < \cdots < \lambda_k,
\]
with corresponding Oseledets spaces $\widehat E_1,\dots,\widehat E_k$.
The same theorem applies to a diffeomorphism $\mathcal R:S\to S$: for an
invariant measure $\mu$ it gives a $D\mathcal R$-invariant splitting
\[
T_x S = \widehat E_1(x) \oplus \cdots \oplus \widehat E_{k(x)}(x)
\]
with Lyapunov exponents
\[
\chi_i(x)=\lim_{n\to\pm\infty} \frac1n \log \|D\mathcal R^n(x)v\| \quad (v\in \widehat E_i(x)\setminus\{0\}).
\]
In this case $\mu$ is a \emph{hyperbolic measure} if all of its Lyapunov
exponents are non-zero; the Oseledets splitting is then the hyperbolic
splitting $E^s\oplus E^u$, with $\widehat E^s$ corresponding to the negative
exponents and $\widehat E^u$ to the positive ones. When $\mathcal R$ is a
Poincar\'e map for an attracting set $\Lambda$, every hyperbolic measure $\mu$
of the flow naturally induces a hyperbolic measure $\nu$ for $\mathcal R$.

\subsection{Symbolic dynamics}
\label{ss:symbolic}

Let $\cA$ be a finite or countable alphabet. We write $\cA^{\N}$ and $\cA^{\Z}$
for the one- and two-sided sequence spaces and $\cT$ for the shift map,
$(\cT x)_i=x_{i+1}$. A \emph{subshift} is a closed $\cT$-invariant set
$\Omega\subset\cA^{\Z}$; it is a \emph{shift of finite type} if $\cA$ is finite and it is defined by
a finite set of forbidden words, and the \emph{full shift} on $\cA$ is
$\Omega=\cA^{\Z}$. For $\omega=(\omega_1,\dots,\omega_m)$ we write $[\omega]$
for the corresponding cylinder, and
$$
\cL_m(\Omega)=\{\omega\in\cA^{m}:[\omega]\cap\Omega\ne\emptyset\}
$$
for the set of admissible words of length $m$, the \emph{language} of $\Omega$
in length $m$. The \emph{separation time} of $x,y\in\Omega$ is the least
$n\ge0$ with $x_n\ne y_n$, and for $\theta\in(0,1)$ the metric
$d_\theta(x,y)=\theta^{\,s(x,y)}$ induces the product topology; a function is
\emph{H\"older} on $\Omega$ if it is Lipschitz with respect to some such
$d_\theta$.

Symbolic models are available for the basic sets of the Poincar\'e map through
Markov partitions. Indeed, let $\widehat K$ be a basic set of a $C^{1+\alpha}$ surface diffeomorphism $g$. Then $\widehat K$ admits a finite Markov partition, and the associated coding
map $\pi:\Omega\to\widehat K$, defined on a shift of finite type $\Omega$, is a
surjective, H\"older, finite-to-one map with $g\circ\pi=\pi\circ\cT$. If the
rectangles of the Markov partition are pairwise disjoint  then $\pi$ is a
homeomorphism, and $g|_{\widehat K}$ is topologically conjugate to
$\cT|_\Omega$.

For $\nu\in\Msr{\cT}{\Omega}$ we write $\pi_*\nu=\nu\circ\pi^{-1}\in
\Msr{g}{\widehat K}$ for the push-forward measure. Conversely, every
$\widehat\mu\in\Msr{g}{\widehat K}$ admits a lift
$\pi^{*}\widehat\mu\in\Msr{\cT}{\Omega}$ with
$\pi_*\pi^{*}\widehat\mu=\widehat\mu$; since $\pi$ is finite-to-one, the lift
preserves entropy,
$h_{\pi^{*}\widehat\mu}(\cT)=h_{\widehat\mu}(g)$, and preserves ergodicity, and
for every continuous $\psi$ one has
$$\int\psi\circ\pi\,d\pi^{*}\widehat\mu=\int\psi\,d\widehat\mu.$$ Finally,
minimality and unique ergodicity are preserved by $\pi$ and by $\pi_*$.
We shall also fix once and for all a metric $\mathfrak d$ on
$\Msr{\cT}{\Omega}$ inducing the weak$^{*}$ topology, so that
$\big(\Msr{\cT}{\Omega},\mathfrak d\big)$ is a compact metric space.

\subsection{Thermodynamic formalism}
\label{ss:thermo}

Let $X_1$ be the time one map of $X$ and $K\subset M$ a compact subset. 
$\psi:K\to\R$ be continuous. Write $S_n\psi=\sum_{j=0}^{n-1}\psi\circ X_{j}$. A
set $E\subset K$ is \emph{$(n,\delta)$-separated} if for all distinct
$x,y\in E$ there is $0\le j<n$ with $d(X_{j}(x),X_{j}(y))>\delta$. The
\emph{topological pressure} of $\psi$ is
$$
P_K(X,\psi)=\lim_{\delta\to0}\ \limsup_{n\to\infty}\frac1n\log\
\sup\Big\{\sum_{x\in E}e^{S_n\psi(x)}\ :\ E\subset K\ \text{is }(n,\delta)
\text{-separated}\Big\},
$$
and the \emph{topological entropy} of $X|_{K}$ is
$h_{\topp}(X|_{K})=P_{K}(X,0)$. We write $P_{\topp}(X,\psi)$ when the ambient
set is clear. 

Assume $K$ is invariant. By the  celebrated variational principle, one has 
$$\displaystyle P_{K}(X,\psi)=\sup\Big\{h_\nu(X)+\int\psi\,d\nu\ :\
\nu\in\Msr{X}{K}\Big\},$$ and the supremum may be taken over
$\Mer{X}{K}$. A measure attaining it is called an \emph{equilibrium state} for
$\psi$; an equilibrium state for $\psi\equiv0$ is a \emph{measure of maximal
entropy}.

\subsection{SRB measures and non-uniform hyperbolicity}
\label{ss:SRB}

\begin{definition}\label{def:SRB}
A measure $\mu\in\Msr{X}{\Lambda}$ is an \emph{SRB measure} if the conditional
measures of $\mu$ along local unstable manifolds are absolutely continuous with
respect to the induced Riemannian volume on unstable leaves; that is, if
$\{\mu^{u}_x\}$ are the conditional measures of a measurable partition
subordinate to local unstable manifolds, then $\mu^{u}_x\ll\Leb^{u}_x$ for
$\mu$-almost every $x\in\Lambda$.
\end{definition}

\begin{remark}\label{rem:SRBGibbs}
For uniformly hyperbolic flows the SRB measures are exactly the equilibrium
states of the \emph{unstable geometric potential}
\begin{equation}\label{eq:psiuX}
\psi^{u}_{X}(x)=-\frac{d}{dt}\Big|_{t=0}\log\big\|DX_t|_{E^{u}_x}\big\| ,
\end{equation}
which satisfies $P_{\topp}\big((X_t)_t|_\Lambda,\psi^{u}_{X}\big)=0$; see
\cite{BR75}.
\end{remark}

\begin{definition}\label{def:NUH}
An attracting set $\Lambda$ is \emph{non-uniformly hyperbolic} if it is
partially hyperbolic and carries an $X$-invariant, ergodic, hyperbolic and fully
supported SRB measure $\mu_\Lambda$. We call $\Lambda$ \emph{non-trivial} if it
is neither a single equilibrium nor a single periodic orbit.
\end{definition}

\begin{remark}
Since $\mu_\Lambda$ is hyperbolic, all of its Lyapunov exponents other than the
one along the flow direction are non-zero, and the Oseledets decomposition
\cite{O68} gives
$$
\limsup_{t\to+\infty}\frac1t\log\|DX_{-t}(x)v\|<0
$$
for $\mu_\Lambda$-almost every $x\in\Lambda$ and every
$v\in E^{cu}_x\setminus\langle X(x)\rangle$.
\end{remark}

\subsection{Hausdorff dimension of sets and of measures}
\label{ss:dimension}

\subsubsection*{Dimension of sets}

Let $(Z,d)$ be a metric space and let $E\subset Z$. For $s\geq0$ and
$\varepsilon>0$, define
\begin{equation}
\mathcal H^s_\varepsilon(E)=\inf\left\{ \sum_{i=1}^{\infty} \diam(V_i)^s: E\subset\bigcup_{i=1}^{\infty}V_i, \quad \diam(V_i)<\varepsilon \right\}.
\end{equation}
The $s$-dimensional Hausdorff measure of $E$ is
$\mathcal H^s(E)=\lim_{\varepsilon\to0}\mathcal H^s_\varepsilon(E)$, and the
\emph{Hausdorff dimension of $E$} is
\begin{equation}
\dim_H(E)=\inf\left\{ s\geq0:\mathcal H^s(E)=0 \right\}=\sup\left\{ s\geq0:\mathcal H^s(E)=+\infty \right\}.
\end{equation}
It is well known that the Hausdorff dimension is preserved by bi-Lipschitz
maps. We refer the reader e.g. to \cite{Falconer,Pe97} for more details.

\subsubsection*{Dimension of measures}

Let $\mu$ be a Borel probability measure on a metric space $Z$. The Hausdorff
dimension of $\mu$ is defined by
\begin{equation}\label{eq:measure-dimension}
\dim_H(\mu)=\inf\left\{ \dim_H(E): E\subset Z \text{ is Borel and } \mu(E)=1 \right\},
\end{equation}
that is, $\dim_H(\mu)$ is the smallest Hausdorff dimension of a set carrying
full $\mu$-measure. A different notion of dimension is given by the local
dimensions. For $x\in Z$, the \emph{lower and upper local dimensions of $\mu$
at $x$} are
\begin{equation*} \underline d_\mu(x)=\liminf_{r\to0} \frac{\log \mu(B(x,r))}{\log r}, \qquad \overline d_\mu(x)=\limsup_{r\to0} \frac{\log \mu(B(x,r))}{\log r}, \end{equation*}
respectively. We say that $\mu$ is \emph{exact dimensional} if there exists
$d\geq0$ such that
\begin{equation*}
d_\mu(x)=\lim_{r\to0} \frac{\log \mu(B(x,r))}{\log r}=d \end{equation*}
for $\mu$-almost every $x$. If $\mu$ is exact dimensional then
$d_\mu(x)=\dim_H(\mu)$ for $\mu$-almost every $x$ (see e.g. \cite{Falconer}).

\subsubsection*{Dimension of hyperbolic measures of a flow}

Let us now focus on the dimension of flow invariant measures. Let $(X_t)_t$ be
a $C^{1+\alpha}$ flow on a three-dimensional manifold, and let $\mu$ be an
ergodic hyperbolic invariant probability measure with Lyapunov exponents
$\lambda^s(\mu)<\lambda^0<\lambda^u(\mu)$, where $\lambda^0=0$ is the exponent
in the flow direction. As the stable and unstable Pesin foliations are
one-dimensional, the stable and unstable conditional measures are exact
dimensional, with dimensions
\begin{equation}\label{eq:transverse-measure-dimensions} \delta^u(\mu)=\frac{h_\mu}{\lambda^u(\mu)}, \qquad \delta^s(\mu)=\frac{h_\mu}{|\lambda^s(\mu)|}, \end{equation}
respectively. In consequence,
\begin{equation}\label{eq:dimension-hyperbolic-measure} \dim_H(\mu)
=1+\delta^s(\mu)+\delta^u(\mu)
= 1+ \frac{h_\mu}{|\lambda^s(\mu)|} + \frac{h_\mu}{\lambda^u(\mu)}, \end{equation}
where $h_\mu$ stands for the metric entropy of $\mu$ with respect to the
time-$1$ map $X_1$, and the term $1$ on the right hand side accounts for the
flow direction (cf. \cite{BPS99,BRW04}).

Accordingly, we set
\begin{equation}\label{eq:def-dim-su}
\dims(\mu)=1+\delta^s(\mu)
\qquad\text{and}\qquad
\dimu(\mu)=1+\delta^u(\mu),
\end{equation}
and likewise $\dims(K)=1+\delta^s(K)$ and $\dimu(K)=1+\delta^u(K)$ for a
horseshoe $K$ of the flow, where $\delta^s(K),\delta^u(K)$ are the transverse
dimensions introduced in Subsection~\ref{ss:transverse}. For a hyperbolic set
$\widehat K$ of the Poincar\'e map we write instead
$\dims(\widehat K)=\delta^s(\widehat K)$ and
$\dimu(\widehat K)=\delta^u(\widehat K)$.

\subsection{Horseshoes and transverse dimensions}
\label{ss:transverse}

Let $K$ be ahyperbolic  horseshoe for the flow $(X_t)_t$. Its tangent bundle admits a continuous
invariant splitting
$T_K M= E^s \oplus \langle X\rangle \oplus E^u,$
where $E^s$ is uniformly contracted and $E^u$ is uniformly expanded. Recall
that a \emph{Markov family} for a hyperbolic flow $(X_t)_t$ on a compact
invariant set $K$ is a finite collection of rectangles
$
\mathcal Q=\{Q_1,\dots,Q_k\}
$
covering $K$ and satisfying the Markov property with respect to the associated
Poincar\'e return map (see e.g. \cite[p. 280]{PS}); it is well known that every
hyperbolic flow admits a Markov family.

\subsubsection*{Transverse dimensions}

Fix a local cross-section $S$ intersecting $K$, let $\mathcal R$ denote the
Poincar\'e first return map, and set
\begin{equation} \widehat K=K\cap S. \end{equation}
By construction, $\widehat K$ is a hyperbolic basic set for the surface local
diffeomorphism $\mathcal R$ on $S$. For $x\in\widehat K$, write
\begin{equation} \widehat K^u(x)=\widehat K\cap W^u_{\mathrm{loc}}(x,\mathcal R), \qquad \widehat K^s(x)=\widehat K\cap W^s_{\mathrm{loc}}(x,\mathcal R). \end{equation}
By \cite{MM83}, the quantities
\begin{equation} \delta^u(K)=\dim_H\bigl(\widehat K^u(x)\bigr) \qquad \text{and} \qquad \delta^s(K)=\dim_H\bigl(\widehat K^s(x)\bigr) \end{equation}
do not depend on the choice of the point $x\in\widehat K$; they are called the
\emph{transverse unstable and stable dimensions} of $K$. Since $K$ is locally
the suspension of the surface horseshoe $\widehat K$, one has
$\dim_H(K)=1+\delta^s(K)+\delta^u(K)$, and the transverse dimensions can be
computed as in the case of surface diffeomorphisms.

Indeed, if $\varphi$ is a $C^{1+\alpha}$ surface diffeomorphism and
$\widehat K$ is a basic hyperbolic set of saddle type, then it follows from
McCluskey and Manning \cite{MM83} that
\begin{equation}\label{eq:MM-dimension-formula} \dim_H(\widehat K)=\delta^s(\widehat K)+\delta^u(\widehat K), \end{equation}
where
\begin{equation} \delta^u(\widehat K)=\dim_H\bigl(\widehat K\cap W^u_{\mathrm{loc}}(x)\bigr) \end{equation}
is independent of $x\in\widehat K$ and is the unique solution of Bowen's
equation $P_{\topp}\bigl(\varphi|_{\widehat K},t\,\psi^u\bigr)=0$, with
$$\psi^u(x)=-\log\left|D\varphi(x)\big|_{E^u_x}\right|;$$ the dimension
$\delta^s(\widehat K)$ is computed likewise with respect to the potential
$$\psi^s(x)=\log\left|D\varphi(x)\big|_{E^s_x}\right|.$$

For our purposes we need a more precise description of the Hausdorff dimension
of a horseshoe, which we take from Pesin and Sadovskaya \cite{PS}. Let
$K\subset\Lambda$ be a hyperbolic horseshoe for the flow $(X_t)_{t\in\mathbb R}$
and fix a Markov family $\mathcal Q=\{Q_1,\dots,Q_k\}$ for the flow on $K$. The
unstable geometric potential $\psi^u_X$ defined in \eqref{eq:psiuX} is
H\"older continuous, and there exists a unique real number $\delta^u(K)$
satisfying Bowen's equation
$$
P_{\topp}
\bigl((X_t)_t|_K,\,
\delta^u(K)\,\psi^u_X
\bigr)=0.
$$
Let $\mu_K$ denote the unique equilibrium state associated to the potential
$\delta^u(K)\psi^u_X$, and denote by
$
(\mu_x^u)_{x\in K}
$
the family of conditional measures of $\mu_K$ with respect to the measurable
partition into local unstable plaques
$
W^u_{\mathrm{loc}}(x)\cap Q_i,
$
with $x\in Q_i$ and $1\le i\le k$. The next theorem ensures that the local
unstable dimension is constant, and hence that it can be computed in a
neighbourhood of any given point of the horseshoe.

\begin{theorem}\cite[Theorem 4.1]{PS}
\label{thm:PS-local-dim}
For every $x\in K$ and every non-empty relatively open subset
$V\subset W_{\mathrm{loc}}^u(x)\cap K$, the following properties hold.

\begin{enumerate}
\item
$
\dim_H(V)
=
\underline{\dim}_B(V)
=
\overline{\dim}_B(V)
=
\delta^u(K).
$
\item
$
d_{\mu_x^u}(y)
=
\delta^u(K)
= \dim_H(\mu_x^u)
$
for $\mu_x^u$-almost every
$y\in W_{\mathrm{loc}}^u(x)\cap K$.
\end{enumerate}
\end{theorem}

\color{black}

\section{Horseshoes with large Hausdorff dimension}\label{sec:ThmA}

This section is devoted to the proof of Theorem~\ref{thm:dimension-flows} on the existence of horseshoes with Hausdorff dimension strictly larger than two for three-dimensional smooth flows modeled by suspensions of Gibbs-Markov maps.  The argument descends from the flow to a one-dimensional model, works there, and comes back and is engineered in two parts. 
First we construct sets $\Gamma\subset J$ whose Hausdorff dimension is arbitrarily close to $1$ at the level of the induced dynamics, and then construct $f$-invariant sets $B\subset I$ with the same property by saturation  (cf. Theorem \ref{thm:key}).
Next, we lift $B$ to a horseshoe $\widehat K$ of the
Poincar\'e map and, since every horseshoe has positive stable transverse
dimension, we show that $\widehat K_\vep$ can be chosen so that
$ 
\delta^s+\delta^u>1.
$ 
Finally, suspending this horseshoe adds one dimension, corresponding to
the flow direction. 

\medskip
Before to start with the proof properly, let us set up the basic content that will be assumed through this entire section.  Here after,  $M$ denotes a three-dimensional compact manifold. Let $X\in \mathfrak{X}^{1+\alpha}(M)$ be a $C^{1+\alpha}$-smooth vector field with an attractor $\Lambda\subset M$. Assume that $(X_t)_t$ is a $C^{1+\alpha}$-smooth flow on $\Lambda$ which is modeled by a suspension flow over a $C^{1+\alpha}$ full-branch, Gibbs--Markov expanding map. 

More precisely, assume that the dynamics of $\Lambda$ is modeled by a Poincaré maps $\mathcal R:\mathcal S\to \mathcal S$ whose  quotient along stable leaves is  a one dimensional map $f:I\to I$ that admits an induced a Gibbs-Markov map $F$ 
admitting an at most countable Markov partition $\mathcal P=\{P_i\}_{i\in\mathcal A}$ of $J$. 
 In addition, $F(x)=f^{R(x)}$, with return time  $R:J\to\N$. 
Define the  total inducing time  $$r_n=\sum_{j=0}^{n-1}R\circ F^{j}$$  and observe that $r_n$
constant on each $J_\omega$, $\omega\in\mathcal G_n$, and in particular
$F^{n}|_{J_\omega}=f^{\,r_n(\omega)}|_{J_\omega}$.

The Markov partition allows us to code orbits of $F$ by sequences in $\mathcal A^{\mathbb N}$, where every finite concatenation of admissible words corresponds to a cylinder set. In particular, every interval of the partition $\mathcal P^{(k)}=\bigvee_{j=0}^{k-1}F^{-j}\mathcal
P$ is labeled by an element $\omega=(\omega_1,\dots,\omega_k)\in\mathcal A^{k}$ we write $J_\omega$ for
the corresponding cylinder of $\mathcal P^{(k)}=\bigvee_{j=0}^{k-1}F^{-j}\mathcal
P$; by the full-branch property, $F^{k}|_{J_\omega}:J_\omega\to J$ is a
bijection. 
Furthermore,
the map $F$ has bounded distortion, i.e,  there is $C_d>0$ such that, for every $k\ge1$, every $\omega\in\mathcal A^{k}$
and all $x,y\in J_\omega$,
$$
\Big|\log\frac{|(F^{k})'(x)|}{|(F^{k})'(y)|}\Big|\le C_d .
$$
Moreover, the Gibbs-Markov property of $F$ gives the existence of a unique invariant probability measure $\mu$ equivalent to Lebesgue with bounded density and positive Lyapunov exponent $\chi_{\mu}$. 

\subsection{Invariant Cantor sets for Gibbs-Markov maps}
We are now ready to start with the proof of Theorem \ref{thm:dimension-flows}.  The first needed ingredient is the following Theorem that allows one to construct invriant cantor sets for $f$ with arbitrarily large Hausdorff dimension. 

\medskip

\begin{theorem}\label{thm:key}
Let $f:I\to I$ be a piecewise $C^{1+\alpha}$ map and suppose
that $f$ admits a full-branch
uniformly expanding Gibbs--Markov map $F:J\to J$ with bounded
distortion. Then there exists an increasing sequence of compact
$f$-invariant hyperbolic Cantor sets
$\Gamma_1\subset\Gamma_2\subset\cdots\subset I$ such that
$$ 
\lim_{N\to\infty}\dim_H(\Gamma_N)=1.
$$ 
Moreover, each $\Gamma_N$ involves only finitely many branches of
the inducing scheme.
\end{theorem}

\begin{proof}

\noindent\textbf{Step 1: Construction of Cantor sets with large dimension for $F^n$.}

Let
$ 
\mathcal P=\{P_i:i\in\mathcal A\}
$ 
be the countable Markov partition of $J$ associated to $F$. By
assumption, for every $i\in\mathcal A$ the restriction
$F|_{P_i}:P_i\longrightarrow J$
is a $C^{1+\alpha}$ diffeomorphism onto $J$, there exists
$\lambda>1$ such that $|F'(x)|\geq\lambda$ 
on every branch, and there exists $C_d>0$ such that
\begin{equation}
    \label{eq:bdd3}
\left|
\log\frac{|(F^n)'(x)|}{|(F^n)'(y)|}
\right|
\leq C_d.
\end{equation}
for every
$n\geq1$, every $n$-cylinder $J_\omega$, $\omega\in\mathcal A^n$,
and every $x,y\in J_\omega$.
We proceed to construct finite-branch subsystems of $F$ that have Hausdorff
dimension arbitrarily close to one.

Fix $0<s<1$. For $n\geq1$ and $\omega\in\mathcal A^n$, set
$$
q_\omega=\frac{|J_\omega|}{|J|}.
$$
Since the $n$-cylinders form a partition of $J$ (mod 0) then 
$$\sum_{\omega\in\mathcal A^n}q_\omega=1.$$ 
On the other hand, since $|(F^n)'|\geq\lambda^n$ on every
$n$-cylinder and $F^n(J_\omega)=J$, the mean value theorem implies that 
$q_\omega\leq\lambda^{-n}$ and, consequently,

\begin{equation}
\begin{aligned}
\sum_{\omega\in\mathcal A^n}q_\omega^s
=
\sum_{\omega\in\mathcal A^n}
q_\omega q_\omega^{s-1} 
\geq
\lambda^{n(1-s)}
\sum_{\omega\in\mathcal A^n}q_\omega
=
\lambda^{n(1-s)}.
\end{aligned}
\label{3.4}
\end{equation}
which tends to infinity as $n\to+\infty$ (recall $s-1<0$).

For each $\omega\in\mathcal A^n$, denote by
$$
h_\omega=(F^n|_{J_\omega})^{-1}:J\longrightarrow J_\omega
$$
the corresponding inverse branch. By the mean value theorem, there
exists $\xi_\omega\in J_\omega$ such that
$ 
|(F^n)'(\xi_\omega)|
=
\frac{|J|}{|J_\omega|}
=
q_\omega^{-1}.
$ 
Hence, by bounded distortion, for every $x\in J$,
\begin{equation}
|h_\omega'(x)|
=
\frac{1}{|(F^n)'(h_\omega(x))|}
\geq e^{-C_d}q_\omega.
\label{3.5}
\end{equation}
Combining \eqref{3.4} and \eqref{3.5}, we obtain
$$
\sum_{\omega\in\mathcal A^n}
\left(\inf_{x\in J}|h_\omega'(x)|\right)^s
\geq
e^{-sC_d}\lambda^{n(1-s)}.
\label{3.6}
$$
Choose $n\ge 1$ sufficiently large so that $e^{-sC_d}\lambda^{n(1-s)}>1$ and extract a finite subset
$\mathcal G:=\mathcal G_n\subset\mathcal A^n$ 
such that
\begin{equation}
\sum_{\omega\in\mathcal G}
\left(\inf_{x\in J}|h_\omega'(x)|\right)^s>1.  
\label{eq:choice1}
\end{equation}
If $B_{\mathcal G}$ denotes the invariant Cantor set 
obtained as maximal invariant set of $F^n$ over $\bigcup_{\omega\in \mathcal G} J_\omega$, then $F^n|_{B_{\mathcal G}}$ is semiconjugate to the full shift
on the finite alphabet $\mathcal G$. For every $k\geq1$, the chain
rule gives
$$
\begin{aligned}
&\sum_{\omega_1,\ldots,\omega_k\in\mathcal G}
\inf_{x\in J}
\left|
(h_{\omega_1}\circ\cdots\circ h_{\omega_k})'(x)
\right|^s
\geq
\sum_{\omega_1,\ldots,\omega_k\in\mathcal G}
\prod_{j=1}^k \inf_J |h_{\omega_j}'|^s
\geq 
\left[
\sum_{\omega\in\mathcal G}
\left(\inf_{x\in J}|h_\omega'(x)|\right)^s
\right]^k.
\end{aligned}
$$
By construction (recall \eqref{eq:choice1}), one concludes that the right hand side above tends to infinity as $k\to+\infty$ and, consequently,
$$ P_{B_{\mathcal G}}
\left(
F^n,-s\log|(F^n)'|
\right)>0.$$ 
Since the pressure function
$$t\mapsto P_{B_{\mathcal G}}
\left(
F^n,-t\log|(F^n)'|
\right)$$ 
is strictly decreasing and its unique zero is
$\dim_H(B_{\mathcal G})$, Bowen's formula yields
$\dim_H(B_{\mathcal G})>s$, as desired. 

\medskip
\noindent\textbf{Step 2: Construction of nested Cantor sets}

We now construct the approximating sequence in a nested way.
Since the alphabet $\mathcal A$ is countable, fix an enumeration
$ \mathcal A=\{i_1,i_2,\ldots\},$ 
and, for every $N\geq 1$, set
$ 
\mathcal A_N=\{i_1,\ldots,i_N\}.
$ 
Denote by $B_N$ the maximal invariant set whose itinerary uses only
symbols in $\mathcal A_N$, that is,
$$
B_N=
\Big\{
x\in J:
F^j(x)\in\bigcup_{i\in\mathcal A_N}P_i
\text{ for every }j\geq0
\Big\}.
$$
By construction one knows that 
\begin{equation}
B_N\subset B_{N+1}, \quad 
\text{for every $N\geq1$.}
\label{eq:nested}
\end{equation}
We claim that
\begin{equation}
\lim_{N\to\infty}\dim_H(B_N)=1.
\label{eq:limBclaim}
\end{equation}
Indeed, by Step~1 for each $0<s<1$  there exist $n\geq1$ and a finite
family
$\mathcal G\subset\mathcal A^n$ such that
$\dim_H(B_{\mathcal G})>s.$ 
Since only finitely many symbols of $\mathcal A$ occur in the words
belonging to $\mathcal G$, there exists $N=N(s)$ such that every
symbol appearing in a word of $\mathcal G$ belongs to $\mathcal A_N$.
Therefore,
$B_{\mathcal G}\subset B_N,$ 
and, consequently, 
$$
\dim_H(B_N)\geq\dim_H(B_{\mathcal G})>s.
$$
This proves the claim.

\medskip
\noindent\textbf{Step 3: Saturation of the Cantor sets}

Let now $R:J\to\mathbb N$ denote the inducing time, so that
$ 
F(x)=f^{R(x)}(x),
$ 
with $R$ constant on each $P_i$. Fix $N\ge 1$. Since $B_N$ involves only
the finitely many domains $P_{a_1},\ldots,P_{a_N}$, we have
\begin{equation}
R_N:=
\max_{1\leq i\leq N}R(P_{a_i})<\infty.
\label{eq:defrn}    
\end{equation}
Define the $f$-saturation of $B_N$ by
$$
\Gamma_N
=
\bigcup_{x\in B_N}
\left\{
f^j(x):0\leq j<R(x)
\right\}.
$$
Since only finitely many inducing branches are involved, all the
intermediate iterates appearing above are continuous on the
corresponding compact pieces. Hence $\Gamma_N$ is compact.
Furthermore, it is simple to check that $\Gamma_N$ is $f$-invariant and, 
by the nesting property \eqref{eq:nested}, that
$\Gamma_N\subset\Gamma_{N+1}.$ 
Moreover, as
$B_N\subset\Gamma_N,$ 
it is clear that 
\begin{equation}
\dim_H(\Gamma_N)\geq\dim_H(B_N)
\quad \text{and}\quad 
\lim_{N\to\infty}\dim_H(\Gamma_N)=1.
\label{eq:limdh}    
\end{equation}

Finally, each $\Gamma_N$ is hyperbolic. Indeed, $F|_{B_N}$ is
uniformly expanding and the inducing times are bounded by
$R_N$. Since only finitely many branches and intermediate
iterates are involved, their derivatives are uniformly bounded away
from zero on the corresponding compact sets. Thus the uniform
expansion of the successive returns under $F$, together with the
uniform bound \eqref{eq:defrn} on the gaps between returns, implies the
existence of constants $C_N>0$ and $\lambda_N>1$ such that
$$
|(f^k)'(x)|\geq C_N\lambda_N^k
$$
for every $x\in\Gamma_N$ and every $k\geq1$. Hence $\Gamma_N$ is a
hyperbolic Cantor set for $f$.
\end{proof}

\subsection{Hyperbolic horseshoes with large dimension for three-dimensional flows.}
Let us now prove Theorem~\ref{thm:dimension-flows}.  
Assume $\Lambda\subset M$ is an attractor for $X$ and that its dynamics is modeled by a full-branch expanding  Gibbs-Markov map and   
let $\mathcal R: S\to S$ denote its Poincar\'e map.

\begin{lemma}
\label{le:dimPoincare}
There exists a $\mathcal R$-invariant horseshoe $\widehat K\subset S$
so that $\dim_H(\widehat K)>1$.
\end{lemma}

\begin{proof}
According to Theorem~\ref{thm:key}, for every $\vep >0$ there exists an $f$-invariant Cantor set $\Gamma_\vep \subset I$ on the interval so that
$\dim_H(\Gamma_\vep)>1-\vep$. 
Furthermore, the monotonicity of the construction ensures that we may assume that $\vep<\vep'$ implies $\Gamma_{\vep'}\subseteq \Gamma_{\vep}$. 
By construction, there exists $N_\vep\ge 1$ so that 
$$
\Gamma_\vep=
\Big\{
x\in J:
F^j(x)\in\bigcup_{i\in\mathcal A_{N_\vep}}P_i
\text{ for every }j\geq0
\Big\}
$$
(recall Step 2 in the proof of Theorem~\ref{thm:key}). Then the set
$$
\widehat K_\vep = 
\Big\{
x\in \mathcal S:
\mathcal R^j(x)\in\bigcup_{i\in\mathcal A_{N_\vep}} (\pi^s)^{-1}(P_i)
\text{ for every }j\geq0
\Big\}
$$
is a $\mathcal R$-invariant horseshoe so that $\pi^s(\widehat K_\vep) \supset \Gamma_\vep$.
By the McCluskey--Manning theorem,
$$
\dim_H(\widehat K_\vep) = \dim_H(\widehat K_\vep^u)+ \dim_H(\widehat K_\vep^s)
$$
while $\vep'<\vep$ implies $\widehat K_{\vep'}\subset \widehat K_{\vep}$, and so
$\vep \mapsto \dim_H(\widehat K_\vep^s)$
is a monotone decreasing map.
Fix $\vep_0>0$ and let $d_0=\dim_H(\widehat K_{\vep_0}^s)>0$. Choosing $0<\vep<\vep_0$ so that
$\dim_H(\widehat K_\vep^u)>1-d_0$ one concludes that
$$
\dim_H(\widehat K_\vep) = \dim_H(\widehat K_\vep^u)+ \dim_H(\widehat K_\vep^s)
    > 1-d_0 + \dim_H(\widehat K_{\vep_0}^s)>1.
$$

\end{proof}

\begin{proof}[Proof of Theorem \ref{thm:dimension-flows}]
    
It is enough to notice that the horseshoe $K$ obtained as suspension of $\widehat K$ has Hausdorff dimension larger than or equal to $\dim_H(\widehat K)+1$, hence larger than 2. This finishes the proof of the theorem. 

\end{proof}

\section{Lagrange and Markov Spectra}\label{sec:markov}

 The goal of this section is to prove Theorem~\ref{cor:Markov-spectrum}, which involves a reduction from continuous-time to discrete-time dynamics. Throughout this section we assume $X\in\mathfrak X^{2}(M)$, $\Lambda$ is a non-trivial attracting set with trapping region $U$, modeled by a full-branch expanding Gibbs--Markov map, and $\Sing(X)\cap\Lambda$ is hyperbolic whenever it is
non-empty.
For every $Y$ sufficiently $C^1$-close to $X$ the set
\begin{equation}
    \label{eq:continuation}
    \Lambda_Y:=\bigcap_{t\ge0}Y_t(U)
\end{equation} 
is again an attracting set, which we call the \emph{continuation} of $\Lambda$. The continuation $\Lambda_Y$ need not be transitive, non-uniformly hyperbolic or modeled by a Gibbs-Markov map. Indeed, we shall not need it. Since each singularity
$\sigma\in\Sing(X)\cap\Lambda$ is hyperbolic, there are finitely many of them, and each one admits a unique hyperbolic continuation $\sigma_Y\in\Sing(Y)\cap\Lambda_Y$. Recall that the continuation  preserves their stable and unstable indices.
After possibly shrinking the neighbourhood of $X$ we may assume that $S$ remains
a global cross-section for every $\Lambda_Y$, so that $Y$ has a Poincar\'e return
map $\cR_Y$ defined on its natural domain in $S$, with $\cR_X=\cR$.

If $\widehat K\subset S$ is a horseshoe for $\cR$, then it admits a unique
hyperbolic continuation $\widehat K_Y$, which is a horseshoe for the
discrete-time system $\cR_Y$. Define its flow-saturation by
\begin{equation}\label{eq:KY}
K_Y:=\big\{Y_t(x):x\in\widehat K_Y,\ 0\le t\le\tau_Y(x)\big\}.
\end{equation}
Since $\widehat K_Y$ is compact, invariant, and disjoint from the
discontinuity set of $\cR_Y$, the return time $\tau_Y$ is continuous
and bounded on $\widehat K_Y$. It follows that $K_Y$ is a compact
$Y$-invariant hyperbolic horseshoe for the flow $(Y_t)_t$.
Moreover, since $K_Y\subset U$ and $K_Y$ is $Y$-invariant, then
\begin{equation}\label{eq:KYinLY}
K_Y=Y_t(K_Y)\subset Y_t(U)
\qquad\text{for every }t\ge0.
\end{equation}
Therefore, it follows that  $K_Y\subset\Lambda_Y.$

\subsection{A key auxiliary result}

The next result provides an \emph{abstract} perturbation theorem for the robustness of large Lagrange and Markov spectra for diffeomorphisms.

\color{black}

\begin{theorem}[Main Theorem in \cite{MR17}]\label{thm:MR17}
Let $\widehat K$ be a horseshoe of a $C^2$ surface diffeomorphism $\varphi:N\to N$ so that $\dim_H(\widehat K)>1$. Then, arbitrarily $C^2$-close to
$\varphi$, there are a diffeomorphism $\varphi_0$ and a $C^2$-open neighbourhood
$\cW$ of $\varphi_0$ such that, for every $\widetilde\varphi\in\cW$, denoting by
$\widehat K_{\widetilde\varphi}$ the continuation of $\widehat K$, there is an
open and dense set $\cH(\widetilde\varphi,\widehat K_{\widetilde\varphi})\subset
C^1(N,\R)$ with
$$
\operatorname{int}\Lspec\big(\widetilde\varphi,\phi,\widehat K_{\widetilde\varphi}\big)\neq\emptyset
\quad\text{and}\quad
\operatorname{int}\Mspec\big(\widetilde\varphi,\phi,\widehat K_{\widetilde\varphi}\big)\neq\emptyset
$$
for every $\phi\in\cH(\widetilde\varphi,\widehat K_{\widetilde\varphi})$.
\end{theorem}

\begin{remark}
\label{rem:local-MR}
Even though Theorem~\ref{thm:MR17} is stated for surface diffeomorphisms, 
its proof only involves the
dynamics in a neighborhood of the horseshoe, through a Markov
partition and the corresponding stable and unstable regular Cantor
sets.
Thus, choosing an open neighborhood $V$ of he Cantor set, 
all the relevant objects, as well as their hyperbolic continuations
under sufficiently small $C^2$ perturbations, remain contained in
$V$. Hence the conclusion of Theorem~\ref{thm:MR17} applies to $C^2$ maps
$
\varphi:V\longrightarrow \varphi(V)\subset N
$
which are diffeomorphisms onto its image.
\end{remark}

Theorem~\ref{thm:MR17} is stated for $C^2$ discrete-time systems, while
our setting concerns flows. The main point is therefore to realize the
perturbation of the Poincar\'e map provided by
Theorem~\ref{thm:MR17} as the return map of a $C^2$-close vector field.
This is the content of Subsection~\ref{subset:realiz} below.

\subsection{Realizing a perturbation of the return map by a perturbation of the field}\label{subset:realiz}

In this subsection we will prove  a perturbation result similar to Theorem~\ref{thm:MR17} in the flow setting. Assume that $S$ is
a global cross-section to the flow ${X_t}\mid_\Lambda$ and $\tau(\cdot)$ denotes the Poincar\'e first return time map, as described in the beginning of the section. Denote $$\mathfrak{W}=W^{s}(\Sing (X))\cup W^{u}(\Sing (X)).$$

\begin{proposition}\label{prop:realization}
Let $V\subset S$ be an open set so that $\overline V$ is disjoint from $\mathfrak{W}$, and assume that 
that $\cR|_{V}$ is a $C^{2}$-diffeomorphism onto its image and $0<\tau_0\le\tau(x)\le\tau_1$, for every  $x\in\overline V$. Let $W\subset\overline W\subset V$ be  an open set.
Then for every $\eta>0$ there is $\rho>0$ with the following property: 
for every $C^{2}$ diffeomorphism $\varphi_0:V\to\varphi_0(V)\subset S$ with
$\|\varphi_0-\cR|_{V}\|_{C^{2}}<\rho$ there are an open set $W\subset V'\subset V$ and a vector field $Y\in\mathfrak X^{2}(M)$
such that
\begin{enumerate}
\item[\rm(i)] $\|Y-X\|_{C^{2}}<\eta$;
\item[\rm(ii)] $ Y=X$ outside a flow box contained in $U$; 
 \item[\rm(iii)] ${\mathcal R}_Y|_W=\phi_0|_W$. 
\end{enumerate}
\end{proposition}

\begin{proof}
Fix a Riemannian metric on $S$ and denote by $\exp$ the corresponding
exponential map. Choose an open set $V'$ such that
$ 
\overline W\subset V'\subset V.
$ 
For $\rho>0$ sufficiently small, every $C^2$ diffeomorphism
$\varphi_0:V\to\varphi_0(V)$ satisfying
$ 
\|\varphi_0-\cR|_V\|_{C^2}<\rho
$ 
also satisfies $\varphi_0(V')\subset\cR(V)$. Hence the map
$$
\psi:=(\cR|_V)^{-1}\circ\varphi_0
$$
is well defined on $V'$ and is a $C^2$ diffeomorphism onto its image.
Moreover,
\begin{equation}\label{eq:psi-close}
\|\psi-\mathrm{id}\|_{C^{2}}=\|\cR^{-1}\circ\varphi_0-\cR^{-1}\circ\cR\|_{C^{2}}\le C_0\,\|\varphi_0-\cR|_{V}\|_{C^{2}},
\end{equation}

with $C_0$ depending only on $\|\cR^{-1}\|_{C^{2}}$ on a compact neighbourhood of $\overline W$. By construction one gets $\cR\circ\psi=\varphi_0$ on $V'$.

\smallskip
Now, recall that by assumption $0<\tau_0\le\tau(x)\le\tau_1$, so we can fix $0<\delta<\tau_0$. Since $\delta<\tau_0$, no point of $V'$ returns to $S$ within time $\delta$. Since $\overline V$ contains no singularity, the map $\Theta(x,t)=X_t(x)$ is a $C^{2}$-embedding of $V'\times[0,\delta]$ into $M$.  Let $\cB=\Theta(V'\times[0,\delta])$ be the corresponding flow box. Notice that, in the
coordinates given by $\Theta$, the field $X$ is the constant vector field $\partial_t$.
\smallskip

Let $0<\gamma\ll \delta$ be small and let $\chi:[0,\delta]\to [0,1]$ be a $C^\infty$ bump function satisfying
\begin{enumerate}
    \item $\chi(s)=0$, for every $s\in [0,\gamma]$;
    \item $\chi(s)=1$ for every $s\in [\delta-\gamma,\delta]$;
    \item $\chi$ is strictly increasing in $[\gamma,\delta-\gamma]$. 
\end{enumerate}
Consider also a $C^\infty$ bump function $\xi: S\to [0,1]$ with 
\begin{equation}
\xi\equiv1 \text{ on } W \text{ and }
\operatorname{supp}\xi\subset V'.
\end{equation}
Since $\psi$ is $C^{2}$-close to the identity, there is a unique $C^{2}$ vector
field $v$ on $V'$, tangent to $S$, with
\begin{equation}\label{eq:v}
  \psi(x)=\exp_x\big(v(x)\big),
  \qquad
  \|v\|_{C^{2}}\le C_0'\,\|\psi-\mathrm{id}\|_{C^{2}} .
\end{equation}
Define, for $t\in[0,\delta]$, the maps
\begin{equation}\label{eq:isotopy}
  \Psi_t:S\to S,\qquad
  \Psi_t(x)=\begin{cases}
    \exp_x\big(\chi(t)\,\xi(x)\,v(x)\big) & ,  x\in V',\\[2pt]
    \hspace{1.4cm} x & , x\notin V' .
  \end{cases}
\end{equation}
As $v$ is $C^2$ close to the null vector field, 
$\Psi_t$ is a $C^{2}$ diffeomorphism of $S$. Moreover, the family is $C^{2}$-smooth in $(x,t)$,
and
$$
\Psi_0=\mathrm{id},\qquad \Psi_\delta=\psi\ \text{ on } W,
$$
while $\Psi_t=\mathrm{id}$ outside $\operatorname{supp}\xi$ for all $t$ close to
$0$. 
Let \begin{equation}\label{eq:generator}
  w_t:=\big(\partial_t\Psi_t\big)\circ\Psi_t^{-1},
\end{equation}
be the time-dependent vector field on $S$ (hence in $V'$) that generates this
isotopy.  Notice that $t\mapsto\Psi_t(x)$ is the integral curve of $w_t$. 

We now define the desired perturbed vector field $Y$, in two steps. First, on $\mathcal C:=V'\times[0,\delta]$ consider the vector field
\begin{equation}\label{eq:Z}
  Z(x,t):=\big(w_t(x),1\big).
\end{equation}
The curve $t\mapsto\big(\Psi_t(x),t\big)$ has velocity
$\big(\partial_t\Psi_t(x),1\big)$, which by \eqref{eq:generator} equals
$Z\big(\Psi_t(x),t\big)$; hence the flow of $Z$ maps $(x,0)$ to
$\big(\Psi_t(x),t\big)$. Moreover $Z=\partial_t$ near $\partial\mathcal C$,
because $w_t$ vanishes there.
Second, we transport $Z$ to $M$ by the embedding $\Theta$ and set
\begin{equation}\label{eq:Xtilde}
  Y:=\Theta_{*}Z \ \text{ on } \cB,
  \qquad
  Y:=X \ \text{ on } M\setminus\cB .
\end{equation}
This is well defined because $\Theta$ is an embedding, so every point of $\cB$
has unique coordinates $(x,t)$. The two definitions agree near $\partial\cB$:
there $Z=\partial_t$, and $\Theta_{*}(\partial_t)=X$ by the very choice of the
flow-box coordinates. Consequently $Y-X$ has compact support in the
interior of $\cB$, and $Y$ is a vector field of class $C^{2}$ on $M$
which coincides with $X$ near $S$.

Since $\Theta$ conjugates the flow of $Z$ with the flow of $Y$ on
$\cB$, the piece of $Y$-orbit of a point $x\in W\subset S$ in time $[0,\delta]$ is
$t\mapsto\Theta\big(\Psi_t(x),t\big)$; in particular it leaves the flowbox at
\begin{equation}\label{eq:exit}
  \Theta\big(\Psi_\delta(x),\delta\big)=\Theta\big(\psi(x),\delta\big)
  =X_\delta\big(\psi(x)\big).
\end{equation}
Furthermore,
\begin{equation}\label{eq:field-estimate}
  \|Y-X\|_{C^{2}}
  \le C_1\sup_{t\in[0,\delta]}\|w_t\|_{C^{2}}
  \le C_1\|\chi'\|_{C^{0}}\|\xi\|_{C^{2}}\|v\|_{C^{2}}
  \le \frac{2C_1C_0'\|\xi\|_{C^{2}}}{\delta}\,\|\psi-\mathrm{id}\|_{C^{2}},
\end{equation}
where $C_1>0$ bounds the distortion of $\Theta$ on the compact set
$\overline{\mathcal C}$. Together with \eqref{eq:psi-close}, the right-hand side
of \eqref{eq:field-estimate} is smaller than $\eta$ once $\rho$ is small enough.

\smallskip
 As the perturbation is supported in
$\cB\subset U\setminus\Sing(X)$, items (i) and (ii) follow.
It remains to compute $\cR_{Y}$ on $W$. For $x\in W$ the
$Y$-orbit of $x$ traverses $\cB$ and leaves it at the point
$\Theta(\Psi_\delta(x),\delta)=X_\delta(\psi(x))$. Outside $\cB$ one has
$Y=X$, so the orbit continues as the $X$-orbit of $X_\delta(\psi(x))$
until it meets $S$, which happens at $\cR(\psi(x))$. It cannot re-enter $\cB$
before doing so: every point of $\cB$ is of the form $X_t(y)$ with $y\in V'
\subset S$ and $0\le t\le\delta$, so re-entering $\cB$ would require crossing
$S$ first. Hence
$$
\cR_{Y}(x)=\cR(\psi(x))=\varphi_0(x),\qquad x\in W,
$$
and (iii) holds after renaming $W$ as $V'$.
\end{proof}

\begin{remark}
   The vector field $Y$  constructed in Proposition \ref{prop:realization} is so that  $\Sing(Y)=\Sing(X)$, that $U$ is still a trapping region for $Y$ and $S$ is a global cross-section for $\Lambda_{Y}$.
\end{remark}

\subsection{Relating discrete and continuous spectra}
\label{sec:relation}

Given a vector field $Y\in \mathfrak X^2(M)$ with a an attracting set $\Lambda$ and a Poincar\'e map $\cR_Y$ and a smooth observable $\phi\in C^{1}(U,\R)$, define
\begin{equation}\label{eq:phihat}
\widehat\phi:V'\to\R,\qquad
\widehat\phi(x)=\max_{0\le t\le\tau_Y(x)}\phi\big(Y_t(x)\big).
\end{equation}
Recall the saturation $K_Y$ of a set 
$\widehat K_Y$
as defined in ~\eqref{eq:KY}.

\begin{lemma}\label{lem:induced-observable}
With the notation above:
\begin{enumerate}
\item[\rm(i)] $\widehat\phi$ is continuous on $\overline{V'}$;
\item[\rm(ii)] for every $x\in K_Y$ there is $x_0\in\widehat K_Y$ with
$$
\limsup_{t\to+\infty}\phi\big(Y_t(x)\big)
=\limsup_{n\to+\infty}\widehat\phi\big(\cR_Y^{\,n}(x_0)\big)
\quad\text{and}\quad
\sup_{t\in\R}\phi\big(Y_t(x)\big)=\sup_{n\in\Z}\widehat\phi\big(\cR_Y^{\,n}(x_0)\big),
$$
so that
$\Lspec\big((Y_t)_t,K_Y,\phi\big)=\Lspec\big(\cR_Y,\widehat\phi,\widehat K_Y\big)$
and likewise for $\Mspec$;
\item[\rm(iii)] the set
$\cH_Y:=\big\{\phi\in C^{1}(U,\R)\ :\ \widehat\phi|_{V'}\in
\cH\big(\cR_Y,\widehat K_Y\big)\big\}$
is open and dense in $C^{1}(U,\R)$.
\end{enumerate}
\end{lemma}

\begin{proof}
It is easy to see that (i) holds, 
since $V'$ is uniformly away from the discontinuity set of $\cR_Y$,
the return time $\tau_Y$ is continuous and bounded on $V'$.

To obtain (ii), write $x=Y_s(x_0)$ with $x_0\in\widehat K_Y$ and
$0\le s\le\tau_Y(x_0)$. Since a time shift does not affect either
the $\limsup$ or the supremum along an orbit, it is enough to consider
$x_0$.  The forward $Y$-orbit of $x_0$ is the
disjoint union of the segments $$\{Y_t(\cR_Y^{\,n}x_0):0\le
t\le\tau_Y(\cR_Y^{\,n}x_0), n\ge0\}.$$ 
By definition of $\widehat\varphi$, the maximum of $\varphi$ along the
$n$-th segment is
$\widehat\varphi\bigl(\cR_Y^n(x_0)\bigr).$ 
Since the return time is bounded away from zero and from infinity on
$\widehat K_Y$, the successive return times tend to infinity.

Therefore,
$$
\limsup_{t\to+\infty}\varphi(Y_t(x_0))
=
\limsup_{n\to+\infty}
\widehat\varphi\bigl(\cR_Y^n(x_0)\bigr).
$$
Since $\cR_Y|_{\widehat K_Y}$ is invertible, the same argument applied
to the full orbit gives
$$
\sup_{t\in\mathbb R}\varphi(Y_t(x_0))
=
\sup_{n\in\mathbb Z}
\widehat\varphi\bigl(\cR_Y^n(x_0)\bigr).
$$
This proves item (ii).

Finally, let us prove (iii). Following the argument in
\cite[Section~4]{MR23}, there exists an open and dense subset
$\mathcal G_Y\subset C^1(U,\mathbb R)$ such that, for every
$\phi\in\mathcal G_Y$, the function
$$\widehat\phi(x) = \max_{0\leq t\leq\tau_Y(x)} \phi(Y_t(x))
$$ 
is of class $C^1$ on a neighbourhood of $\widehat K_Y$. Moreover, the
map
$$
\mathcal T_Y:\mathcal G_Y\longrightarrow C^1(V',\mathbb R),
\qquad
\mathcal T_Y(\phi)=\widehat\phi,
$$
is a $C^1$ submersion, hence $\mathcal T_Y$ is locally open. Since
$\mathcal H(\cR_Y,\widehat K_Y)$ is open and dense in
$C^1(V',\mathbb R)$, it follows that
$$ 
\mathcal H_Y
:=
\{
\phi\in\mathcal G_Y:
\widehat\phi\in
\mathcal H(\cR_Y,\widehat K_Y)
\}
$$ 
is open and dense in $\mathcal G_Y$, and hence in
$C^1(U,\mathbb R)$. This finishes the proof of item (iii).
\color{black}
\end{proof}

The last ingredient we need is the following simple monotonicity fact.

\begin{lemma}\label{lem:spectra-elementary}
Let $(Y_t)_t$ be a flow, $\phi$ be a continuous observable, and let
$\Delta\subset\Delta'$ be compact $(Y_t)_t$-invariant sets. Then
$$
\Lspec\big((Y_t)_t,\Delta,\phi\big)\subseteq\Lspec\big((Y_t)_t,\Delta',\phi\big),
\qquad
\Mspec\big((Y_t)_t,\Delta,\phi\big)\subseteq\Mspec\big((Y_t)_t,\Delta',\phi\big),
$$
$$ 
\Lspec\big((Y_t)_t,\Delta,\phi\big)\subseteq
\Mspec\big((Y_t)_t,\Delta,\phi\big).$$
\end{lemma}

\begin{proof}
The first two inclusions are immediate, since the spectra are sets of
values indexed by the points of the invariant set. For the third, let
$$\ell=\limsup_{t\to+\infty}\phi(Y_t (x)), \text{ with } x\in\Delta,$$ and, by compactness of $\Delta$, choose
$t_n\to+\infty$ with $ \phi(Y_{t_n}(x))\to\ell $ and $Y_{t_n}(x)\to y.$  
By continuity of $\phi$ and $(Y_t)_t$-invariance of $\Delta$ it follows that  $\phi(y)=\ell$, and for every
$s\in\R$ one has $$\phi(Y_s (y))=\lim_{n\to\infty}\phi(Y_{t_n+s}(x))\le\ell,$$ because
$\limsup_{t\to+\infty}\phi(Y_t(x))=\ell.$  Hence
$\sup_{s}\phi(Y_s (y))=\ell$, as desired to prove.
\end{proof}


\subsection{Proof of Theorem \ref{cor:Markov-spectrum}}

The proof of the theorems builds over the relation between Lagrange and Markov spectra for continuous flows and Poincar\'e maps obtained in Subsection~\ref{sec:relation}.

Fix $\vep>0$. By Theorem~\ref{thm:dimension-flows}  there are a horseshoe $K\subset\Lambda$ with
$\dim_H(K)>2$, as well as a horseshoe $\widehat K=K\cap S$ for the Poincar\'e map $\cR$ satisfying
$\dim_H(\widehat K)>1$. Recall that $\widehat K$ is
disjoint from $W^{s}(\Sing X)\cup W^{u}(\Sing X)$, so we may choose an open set $ W\subset \overline W\subset V\subset S$ with $\widehat K\subset W$ and $\overline V$ compact and
disjoint from those sets. On $V$, the map $\cR$ is a $C^{2}$ diffeomorphism onto its image and $\tau$ is bounded away from zero and infinity.

 Apply Theorem~\ref{thm:MR17} to the local diffeomorphism
$\varphi=\cR|_{V}$ on the surface $N=S$
and to the horseshoe $\widehat K\Subset V$
(cf. Remark~\ref{rem:local-MR}).
It produces $\varphi_0$
arbitrarily $C^{2}$-close to $\cR|_{V}$ and a $C^{2}$-neighbourhood $\cW$ of
$\varphi_0$ with the stated genericity property.
Now, let $\rho>0$ be given by
Proposition~\ref{prop:realization} (applied with $\eta=\vep$).  Choosing  $\varphi_0$ with
$\|\varphi_0-\cR|_{V}\|_{C^{2}}<\rho$, we obtain an open $V'\supset W$ and a vector field $Y\in\mathfrak X^{2}(M)$ with
$\|Y-X\|_{C^{2}}<\vep$ and $\cR_{Y}|_{V'}=\varphi_0|_{V'}$.
In particular $\cR_{Y}|_{V'}\in\cW$.  
The map $\widetilde Y\mapsto\cR_{\widetilde Y}|_{V'}$ is continuous in the $C^{2}$ topology on a
neighbourhood of $Y$ in $\mathfrak X^{2}(M)$, because $\overline{V'}$
is at a positive distance to the singular set and $\tau$ is bounded.
Hence there is an open $\cU\subset\mathfrak X^{2}(M)$ with
$Y\in\cU$ such that $\cR_{\widetilde Y}|_{V'}\in\cW$ for every $\widetilde Y\in\cU$.

 Fix $\widetilde Y\in\cU$. Since
$\cR_{\widetilde Y}|_{V'}\in\cW$, Theorem~\ref{thm:MR17} provides the continuation
$\widehat K_{\widetilde Y}\subset V'$ and the open dense set
$\cH(\cR_{\widetilde Y},\widehat K_{\widetilde Y})\subset C^{1}(V',\R)$ such that
\begin{equation}\label{eq:spectra}
\operatorname{int}\Lspec\big(\cR_{\widetilde Y},\widehat\phi,\widehat K_{\widetilde Y}\big)\neq\emptyset
\end{equation}
for every $\widehat\phi\in\cH(\cR_{\widetilde Y},\widehat K_{\widetilde Y})$. Let $\cH_Y$ be as in
Lemma~\ref{lem:induced-observable}(iii), an open and dense subset of
$C^{1}(U,\R)$, and let $\phi\in\cH_{\widetilde Y}$. By
Lemma~\ref{lem:induced-observable}(ii) and \eqref{eq:spectra},
$$
\operatorname{int}\Lspec\big(({\widetilde Y}_t)_t,K_{\widetilde Y},\phi\big)
=\operatorname{int}\Lspec\big(\cR_{\widetilde Y},\widehat\phi,\widehat K_{\widetilde Y}\big)\neq\emptyset .
$$
By \eqref{eq:KYinLY} and Lemma~\ref{lem:spectra-elementary},
$$
\Lspec\big(({\widetilde Y}_t)_t,K_{\widetilde Y},\phi\big)\subset\Lspec\big(({\widetilde Y}_t)_t,\Lambda_{\widetilde Y},\phi\big)
\subset\Mspec\big(({\widetilde Y}_t)_t,\Lambda_{\widetilde Y},\phi\big),
$$
so both $\Lspec(({\widetilde Y}_t)_t,\Lambda_{\widetilde Y},\phi)$ and
$\Mspec(({\widetilde Y}_t)_t,\Lambda_{\widetilde Y},\phi)$ have non-empty interior. This finishes the proof of Theorem~\ref{cor:Markov-spectrum}.
\hfill $\square$

\begin{remark}\label{rem:robust}
The conclusion is robust although the hypothesis need not be: we do not
claim that $\Lambda_Y$ is still modeled by a full-branch expanding
Gibbs--Markov map, and this is not needed, since what persists under
perturbation is the horseshoe $\widehat K_Y$ and not the modeling structure.
The Gibbs--Markov hypothesis is used only once, in Step~1, to produce
$\widehat K$.
\end{remark}

\section{Intermediate dimension properties}\label{sec:interdiate}

In this section we investigate the flexibility of the Hausdorff dimension spectrum and prove Theorems~\ref{thm:intermediate-metric-dimensions} and~\ref{thm:intermediate-dimensions}. 
Throughout, we assume that the vector field $X$ and the attracting set $\Lambda$ satisfy the hypotheses of these theorems.

\subsection{Proof of Theorem \ref{thm:intermediate-metric-dimensions}}
The proof of the theorem explores the techniques developed in \cite{DHT22} for diffeomorphisms on closed manifolds, adapted here to the analysis of Poincaré return maps. 
\color{black}
Fix $c\in [1,\gamma)$. 
 If $c=1$, choose a periodic orbit $\mathcal O$ contained in a
horseshoe $K\subset\Lambda$, whose existence follows from
Theorem~\ref{thm:dimension-flows}. The invariant probability measure $\mu_{\mathcal O}$
supported on $\mathcal O$ satisfies
$\dim_H(\mu_{\mathcal O})=1.$ 
 \smallskip

 Now assume that $c\in (1,\gamma)$. Let $\mathcal R:S\to S$ be a  global Poincar\'e map for $\Lambda$. 
  If $K\subset \Lambda$ is a hyperbolic horseshoe, the set $\widehat K=K\cap S$ is a hyperbolic horseshoe for $\mathcal R$. 
  Moreover, if $\tau:\widehat K\to (0,\infty)$ denotes the Poincar\'e first return time, the restriction of the flow to $K$ is conjugate to the suspension flow of $\mathcal R:\widehat K\to \widehat K$ with roof function $\tau$.
  
  Furthermore, there is a one-to-one map   $\Xi:\Msr{\mathcal R}{\widehat K}\to \Msr{X}{K}$  which maps ergodic measures into ergodic measures and in local coordinates corresponds to $$\widehat\mu \mapsto \frac{\widehat\mu \times \Leb}{\int \tau \, d\widehat \mu}$$
and, 
in particular, 
$\dim_H(\Xi(\hat \mu))=\dim_H(\hat\mu)+1$
(see e.g. ~\cite{BaBook,BR75} for more details).

\smallskip
  Now observe that, 
  since $\widehat K$ is a compact invariant subset of the domain of $\mathcal R$  uniformly away from the discontinuity set of $\mathcal R$, there exists an open neighborhood $V$ of $\widehat K$ such that $\mathcal R|_{V}:V\to \mathcal R(V)$ is a $C^1$-diffeomorphism.  In particular, $\widehat K$ may be regarded as a hyperbolic horseshoe for the diffeomorphism $\mathcal R|_{V}$,  hence the results of  \cite{DHT22} apply. 
  
Now, note that 
$$
\widehat\gamma:=\sup\Big\{
\dim_H(\widehat\mu): \widehat\mu\in \Msr{\mathcal R}{\widehat K}
\Big\}=\gamma-1.
$$

 Observe that  $c-1\in [0, \gamma-1)$ and there is a hyperbolic horseshoe $\widehat K$ for $\cR$ such that $c-1\in (0, \widehat\gamma)$, 
Since $\widehat K$ is a locally maximal  hyperbolic set for $\cR$, it is transitive, expansive and has the shadowing property. In addition, it is conformal due to the dimension of $V$. 
Because ${\widehat K}$ is average conformal and $\cR$ is $C^{1+\alpha}$, the Young formula gives, for every $\widehat \mu \in \cM_{\cR}^e({\widehat K})$,
\begin{equation}\label{eq:flex-young}
\dim_H \widehat \mu = \frac{h_{\widehat \mu}({\cR})}{\chi_u({\widehat \mu})} - \frac{h_{\widehat \mu}({\cR})}{\chi_s({\widehat \mu})},
\end{equation}
where $\chi_u({\widehat \mu}) > 0 > \chi_s({\widehat \mu})$ are the Lyapunov exponents of ${\widehat \mu}$.
Define the potentials  on ${\widehat K}$ by
$$
\psi^u(x) := \log \|D{\cR}|_{E^u_x}\|,
\qquad \psi^s(x) := \log \|D{\cR}|_{E^s_x}\|.
$$
By Birkhoff's ergodic theorem,
$$
\int_{{\widehat K}} \psi^u \, d{\widehat \mu} = \chi_u({\widehat \mu}), \qquad 
\int_{{\widehat K}} \psi^s \, d{\widehat \mu} = \chi_s({\widehat \mu}).
$$
Consider the map
$$
\mathcal{T}_{\psi^u,\psi^s,{\cR}} : \cM_{\cR}({\widehat K}) \to \mathbb{R}^3, \qquad
{\widehat \mu} \mapsto \left( \int \psi^u \, d{\widehat \mu},\; \int \psi^s \, d{\widehat \mu},\; h_{\widehat \mu}({\cR}) \right).
$$
and the function
$$
Q(x,y,z) := \frac{z}{x} - \frac{z}{y}.
$$
Notice that is continuous on the domain $\{x>0,\; y<0\}$, which in turn contains the image of $\mathcal{T}_{\psi^u,\psi^s,{\cR}}$.
Using \eqref{eq:flex-young},
$$
\dim_H {\widehat \mu} = Q\!\left( \int \psi^u \, d{\widehat \mu},\; \int \psi^s \, d{\widehat \mu},\; h_{\widehat \mu}({\cR}) \right).
$$
By \cite[Lemma~7.1(c)]{DHT22}, applied to the transitive topologically
Anosov system $({\widehat K},{\cR}|_{\widehat K})$ and to the continuous
function $Q$, the set
$$
Q\bigl(\mathcal{T}_{\psi^u,\psi^s,\cR}
(\cM_{\cR}^e(\widehat K))\bigr)
=
\Big\{\dim_H\widehat\mu:
\widehat\mu\in\cM_{\cR}^e(\widehat K)\Big\}
$$
is an interval. Moreover, this interval contains $0$, since any
periodic orbit contained in $\widehat K$ supports an ergodic invariant
measure of Hausdorff dimension zero.
Hence, for every $\nu\in\cM_{\cR}^e(\widehat K)$, the interval above
contains the entire segment
$[0,\dim_H\nu].$ 
Taking the supremum over all such measures, we obtain
$$
[0,\hat\gamma)
\subset
\Big\{\dim_H\widehat\mu:
\widehat\mu\in\cM_{\cR}^e(\widehat K)\Big\}.
$$
Since $c-1<\hat\gamma$, there exists
$\widehat\mu\in\cM_{\cR}^e(\widehat K)$ such that
$\dim_H(\widehat\mu)=c-1.$ 
Setting $\mu_c=\Xi(\widehat\mu)$, we finally obtain
$\dim_H(\mu_c)=
\dim_H(\widehat\mu)+1
=
c,
$ 
which completes the proof.
\hfill $\square$

\subsection{Proof of Theorem \ref{thm:intermediate-dimensions}}

As explained in the introduction, the proof of Theorem \ref{thm:intermediate-dimensions} is substantially more delicate than the one of Theorem~\ref{thm:intermediate-metric-dimensions}, and  
the measures produced in Theorem~\ref{thm:intermediate-metric-dimensions} are of no help here: in the absence of a variational principle for the Hausdorff dimension, the dimension of a measure gives no control over the dimension of its support (cf. \cite{MM83}). 
Our strategy is therefore to 
adapt methods from \cite{AOR25}, which allows one to control simultaneously the entropy and the geometric potentials of certain Poincar\'e maps, in order to construct subsystems enjoying a stronger property, namely strict ergodicity, for which the two quantities agree. 

\smallskip
Let us prove the statement in case of  the unstable dimensions (the proof for the stable dimensions is analogous). Fix $c\in [1,\gamma^u)$. By definition of $\gamma^u$, there exists a hyperbolic horseshoe $K\subset \Lambda$ such that $$1\leq c< \dimu(K).$$ 

Let $\mathcal R:S\to S$ be a  global Poincar\'e map for $\Lambda$ and recall that, since $K$ is a hyperbolic horseshoe, the set $\widehat K=K\cap S$ is a hyperbolic horseshoe for $\mathcal R$.  Consequently, the unstable geometric potential $\psi^u_{\mathcal R}$ of the Poincar\'e map is
 naturally related to the unstable geometric potential $\psi^u_X$ of the flow by
$$
\psi_{\mathcal R}^u(x)
=
\int_0^{\tau(x)}
\psi_X^u(X_t(x))\,dt.
$$
Observe that,  up t considering an adapted metric , we may assume that $-\psi^u_{\mathcal R}$ is a strictly positive potential for $\mathcal R$. Since $K$ is the suspension flow of $\widehat K$, it follows that
$$\dimu(\widehat K)=\dimu(K)-1.$$

On the other hand, since $\widehat K$ is a  horseshoe, it is topologically conjugated to a subshift of finite type $\mathcal T:\Omega\to\Omega$, i.e., there is a homeomorphism $\pi:\Omega\to \widehat K$ such that $\mathcal R\circ \pi=\pi\circ \mathcal T$. Hence, one can define a strictly positive continuous potential $\widetilde\psi^{\,u}\in {C}^0(\Omega)$  by
$$\widetilde\psi^{\,u}(x)=-\psi^u_{\mathcal R}(\pi(x)), \text{ for every } x\in \Omega.$$ Observe that $$\int\widetilde\psi^{\,u}  \,d\pi^*\widehat\mu=\int -\psi^u_{\mathcal R}\,d\widehat\mu, \textrm{ for every } \widehat\mu\in \Msr{\mathcal R}{\widehat K},$$
where $\pi^*$ denotes the lift of measures through the coding $\pi$ and $\pi_*$ its push-forward.

\smallskip
Now, let $\widehat\mu\in \Msr{\mathcal R}{\widehat K}$ be any ergodic measure satisfying $$(c-1)< h^*=-\frac{h_{\widehat\mu}(\mathcal R)}{\int\psi^u_{\mathcal R}\,d\widehat\mu}=\frac{h_{\pi^*\widehat\mu}(\mathcal T)}{\int\widetilde\psi^{\,u} \,d\pi^*\widehat\mu}.$$
Such a measure can be chosen e.g. as an equilibrium state for the potential $\delta^u(\widehat K)\psi^u_{\mathcal R}$, as for this potential one has  $$h^*=\delta^u(\widehat K)=\dim^u_H(K)-1>c-1.$$

Let $\mathfrak{d}$ denote a metric compatible with the weak*-topology of $\Msr{\mathcal T}{\Omega}$. By \cite[Theorem 3.2]{AOR25},   there exist  a sequence $(\Omega_n)_{n\geq 0}$ of subshifts of finite type and two strictly decreasing sequences of positive numbers $(\varrho_i)_{i\geq 1}$ and $(\kappa_i)_{i\geq 1}$   such that, by denoting $\nu_0=\pi^*\widehat\mu$ and $\nu_n$ the measure of maximal entropy of $\Omega_n$, it holds:

\begin{enumerate}[(i)]
    \item $\varrho_n,\kappa_n\to0$, $\Omega_0=\Omega$ and $\Omega_n\supset \Omega_{n+1}$, for every $n\geq 0$.
    \item For every $n\geq 0$ and for every measure $\zeta\in \Msr{\mathcal T}{\Omega_n}$, we have  $\mathfrak{d}(\zeta,\nu_n)\leq 2\kappa_n$,
\begin{equation}
(1-\varrho_{n+1})\int \widetilde\psi^{\,u} \,d\nu_n < \int \widetilde\psi^{\,u} \,d\zeta < (1+\varrho_{n+1})\int \widetilde\psi^{\,u} \,d\nu_n.
\end{equation} and
\begin{equation}
(1+\varrho_{n+1})^2 (c-1)\int \widetilde\psi^{\,u} \,d\nu_n \leq h_{\topp}(\Omega_{n+1})
\leq
(1+3\varrho_{n+1})(c-1)\int \widetilde\psi^{\,u} \,d\nu_n.
\end{equation}

\item there exists an increasing sequence of non-negative integers $(m_n)_n$ such that $\mathcal{L}_{m_n}(\Omega_i)=\mathcal{L}_{m_n}(\Omega_n)$ for every $i\geq n$
and for every word $w\in \mathcal{L}_{m_n}(\Omega_n)$ and every $v\in \mathcal{L}_{m_j}(\Omega_n)$ with $j<n$, $v$ is a subword of $w$.
\end{enumerate}

\smallskip
Consider the set $$\Omega_c=\bigcap_{i\geq 0} \Omega_i$$

\begin{claim}
    \emph{The set $\Omega_c$ is strictly ergodic and satisfies $$c-1=\frac{h_{\nu_c}(\mathcal T)}{\int\widetilde\psi^{\,u} \,d\nu_c},  $$
where $\nu_c$ is the unique ergodic measure of $\Omega_c$.}
\end{claim}
\begin{proof}
Since the set $\Omega_c$ is the intersection of nested compact and invariant sets, then it is non-empty, compact and invariant. Moreover, we have  $\mathfrak{d}(\nu,\mu)<4\kappa_n$, for every $n\geq 0$ and any pair of invariant measures of $\mathcal T|_{\Omega_c}$. Hence, item (i) gives $\mathfrak{d}(\mu,\nu)=0$ and hence $\mu=\nu$, proving that $\Omega_c$ is uniquely ergodic. Moreover, the minimality of $\Omega_c$ is a direct consequence of  (iii), proving that it is strictly ergodic.

Finally, let $\nu_c$ be   the unique probability measure in $\Msr{\mathcal T}{\Omega_c}$. On the one hand,  $\nu_c\in \Msr{\mathcal T}{\Omega_n}$ for every $n$, and hence, condition (ii) gives:
$$
\frac{h_{\nu_c}(\mathcal T)}{\int \widetilde\psi^{\,u}  \,d\nu_c}\leq \frac{h_{\topp}(\Omega_{n})}{(1-\varrho_n)\int \widetilde\psi^{\,u}  \,d \nu_{n-1}}\leq (c-1)\frac{1+3\varrho_n}{1-\varrho_n}\longrightarrow c-1.
$$

On the other hand, we have $\nu_n\to \nu_c$. So,
$$
h_{\nu_c}(\mathcal T)=h_{\topp}(\Omega_c) {\geq}\limsup_{n\to\infty} h_{\nu_n}(\mathcal T)=
\limsup_{n\to\infty} h_{\topp}(\Omega_n)\geq \limsup_{n\to\infty} \; (c-1)\int \widetilde\psi^{\,u}  \,d\nu_n,
$$
because the entropy function is upper semi-continuous in $\Omega$.
Therefore for every $\eta>0$ and infinitely many large values of $n$ one has 
$$
h_{\nu_c}(\mathcal T) \geq (c-1)\int \widetilde\psi^{\,u}  \,d\nu_n-\eta.
$$
Using (ii) once more, it follows that 
$$
\frac{h_{\nu_c}(\mathcal T)}{\int \widetilde\psi^{\,u}  \,d\nu_c}\geq\lim\limits_{\eta\to 0}\lim\limits_{n\to \infty}\frac{c-1}{(1+\varrho_{n+1})}-\frac{\eta}{\int \widetilde\psi^{\,u}  \,d\nu_c
}=c-1,
$$
thereby proving the claim.
\end{proof}

Let us proceed with the proof of the theorem.
Define $\widehat\mu_c=\pi_*\nu_c$. Since $\Omega_c$ is strictly ergodic and $\pi$ is a conjugacy, then $\widehat K_c=\pi(\Omega_c)$ is also a strictly ergodic subsystem of $\widehat K$.  Moreover, we have
$$c-1=\frac{h_{\topp}(\mathcal T|_{\Omega_c})}{\int\widetilde\psi^{\,u} \,d\nu_c}= \frac{h_{\nu_c}(\mathcal T)}{\int\widetilde\psi^{\,u} \,d\nu_c}=-\frac{h_{\widehat\mu_c}(\mathcal R)}{\int\psi^u_{\mathcal R}\,d\widehat\mu_c}.$$
On one hand, the variational principle combined with the Ledrappier--Young formula for unstable Hausdorff dimension gives
\begin{equation}\label{eq:flex1}
c-1=\dimu(\widehat\mu_c)=-\frac{h_{\widehat\mu_c}(\mathcal R)}{\int\psi^u_{\mathcal R}\,d\widehat\mu_c}=-\frac{h_{\topp}(\mathcal R|_{\widehat K_c})}{\int\psi^u_{\mathcal R}\,d\widehat\mu_c},
    \end{equation}
because  $\widehat\mu_c$ is the unique invariant measure on $\widehat K_c$. On the other hand,  Bowen's formula for the Hausdorff dimension of the unstable projection of an invariant set gives
\begin{equation}\label{eq:flex2}
    \dimu(\widehat K_c) = -\frac{h_{\topp}(\pi^u(\widehat K_c))}{\int\psi^u_{\mathcal R} \,d\widehat\mu_c}.
\end{equation}
However, one has
\begin{equation}\label{eq:flex3}
h_{\topp}(\pi^u(\widehat K_c)) = h_{\topp}(\widehat K_c) = h_{\widehat\mu_c}(\mathcal R),
\end{equation}
because $\pi^u$ preserves topological entropy. Comparing \eqref{eq:flex1}, \eqref{eq:flex2} and \eqref{eq:flex3}, we obtain
$$
\dimu(\widehat K_c) = \dimu(\widehat\mu_c).
$$
Finally, by considering the suspension $K_c$ of the map $\mathcal R:\widehat K_c\to \widehat K_c$ with roof function $\tau$, we get that $$\dimu(K_c)= \dimu(\widehat K_c)+1=c$$
To complete the proof observe that the minimality of $K_c$ is a direct consequence of the minimality of $\widehat K_c$, while its unique ergodicity comes from the unique  ergodicity of $\widehat K_c$ combined with the bijective property of of the map $\Xi$.
\hfill $\square$

\section{Applications to non-uniformly hyperbolic flows}\label{sec:applications}

In this section we discuss several classes of three-dimensional
partially hyperbolic attractors for which the hypotheses of the main results are satisfied. Consequently, for each class considered
below, one obtains:

\begin{itemize}
    \item[(1)] the existence of hyperbolic horseshoes of Hausdorff dimension
    strictly larger than $2$;
    \item[(2)] robust nonempty interior for the corresponding dynamical
    Lagrange and Markov spectra, for a $C^1$-open and dense set of
    observables, under the $C^2$ regularity assumption required in
    Theorem~\ref{cor:Markov-spectrum};
    \item[(3)] the realization of all intermediate Hausdorff dimensions
    of ergodic invariant probability measures in the range prescribed
    by Theorem~\ref{thm:intermediate-metric-dimensions};
    \item[(4)] the realization of all intermediate stable and unstable
    Hausdorff dimensions by strictly ergodic hyperbolic subsets, in
    the ranges prescribed by Theorem~\ref{thm:intermediate-dimensions}.
\end{itemize}
For each of the classes of three-dimensional partially hyperbolic attractors one needs to check: (i)
existence of a full-branch uniformly expanding
    Gibbs--Markov induced map for the quotient dynamics;
(ii) the hyperbolic and symbolic structures required in
    Theorems~\ref{thm:intermediate-metric-dimensions} and ~\ref{thm:intermediate-dimensions};
and (iii) the additional $C^2$ regularity required in Theorem~\ref{cor:Markov-spectrum}.

\subsection{Singular-hyperbolic sets and Lorenz attractors}

The geometric Lorenz attractor was introduced, independently, in  \cite{ABS} and \cite{G}  as a robust geometric model for Lorenz's polynomial system
\begin{displaymath}
\left\lbrace \begin{array}{lll}
\dot{x}&=& \sigma(y-x) \\
\dot{y}&=& \rho x-y-xz \\
\dot{z}&=& xy-\beta z,
\end{array}\right.
\end{displaymath}
where $\sigma\approx 10$, $\beta\approx 8/3$ and $\rho\approx28$.

This model presents very rich dynamics and shares several chaotic properties, but it fails to be hyperbolic due to the existence of a hyperbolic singularity accumulated by regular orbits within the attractor.  A remarkable structural property of the geometric Lorenz attractors is that its singularity has three real eigenvalues $\lambda_{ss}\leq\lambda_s<0<\lambda_u$ satisfying $\lambda_u+\lambda_s>0$. Moreover, the geometric Lorenz attractor admits a global cross-section $S$ and a Poincar\'e return map $\mathcal R:S\to S$ with a one-dimensional contracting foliation (see Figure~\ref{fig:LORENZ} below for a geometric representation). After taking the quotient along the  stable foliation, one obtains a one-dimensional map
\begin{equation}
    \label{eq:f1im}
    f:I\setminus\{0\}\to I, \qquad I=[-1,1],
\end{equation}
with a discontinuity at $0$.
The geometric Lorenz attractor is one most paradigmatic example of singular-hyperbolic sets, which we now define.

\begin{figure}[h]
    \centering
    \begin{minipage}[t]{0.48\textwidth}
        \centering
        \includegraphics[width=\textwidth]{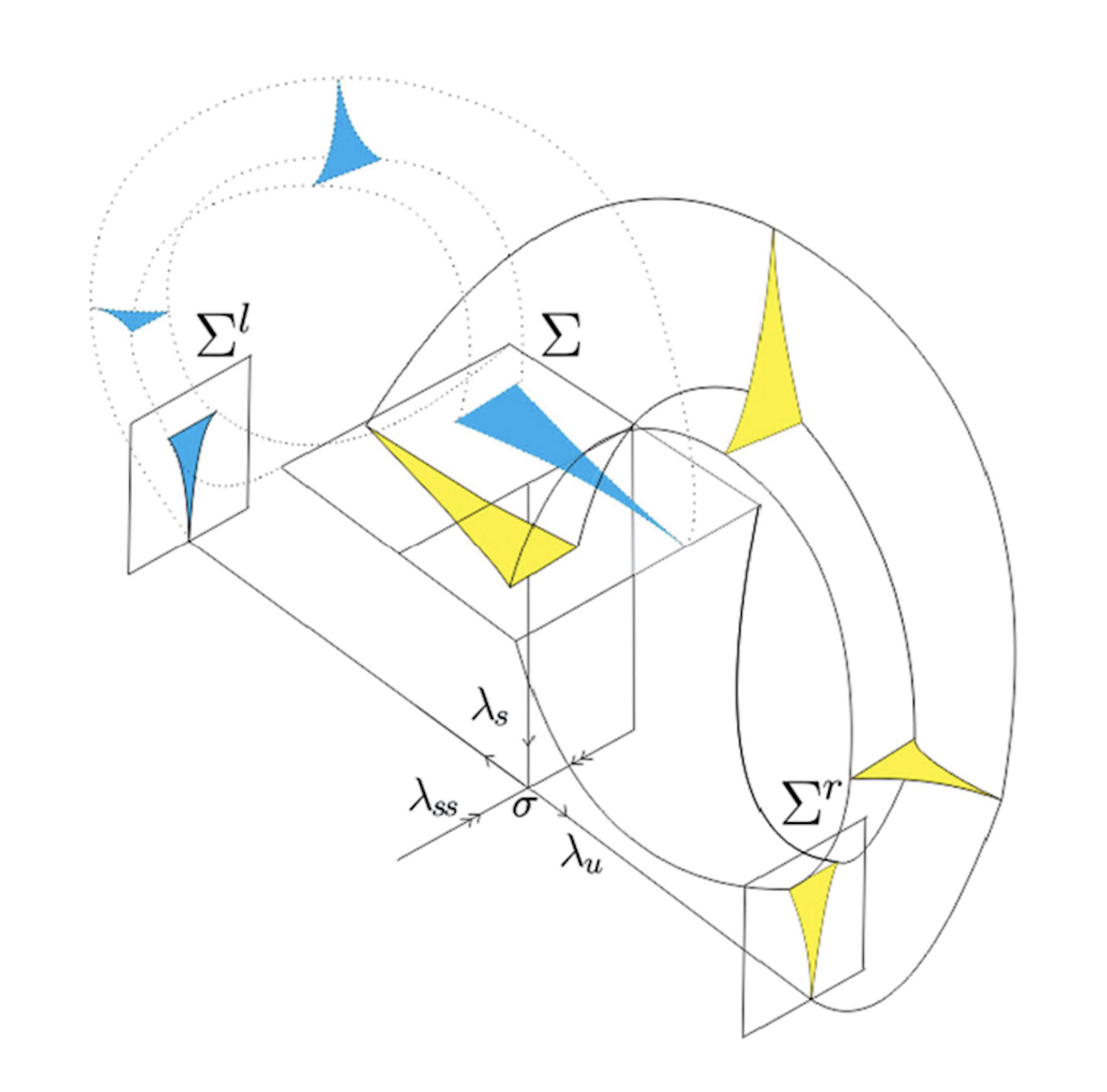}
        \caption{The Geometric Lorenz Attractor:
$$
\lambda_{ss}<\lambda_s<0<\lambda_u
$$
$$
\lambda_s+\lambda_u>0
$$}
        \label{fig:LORENZ}
    \end{minipage}
    \hfill
    \begin{minipage}[t]{0.48\textwidth}
        \centering
        \includegraphics[width=\textwidth]{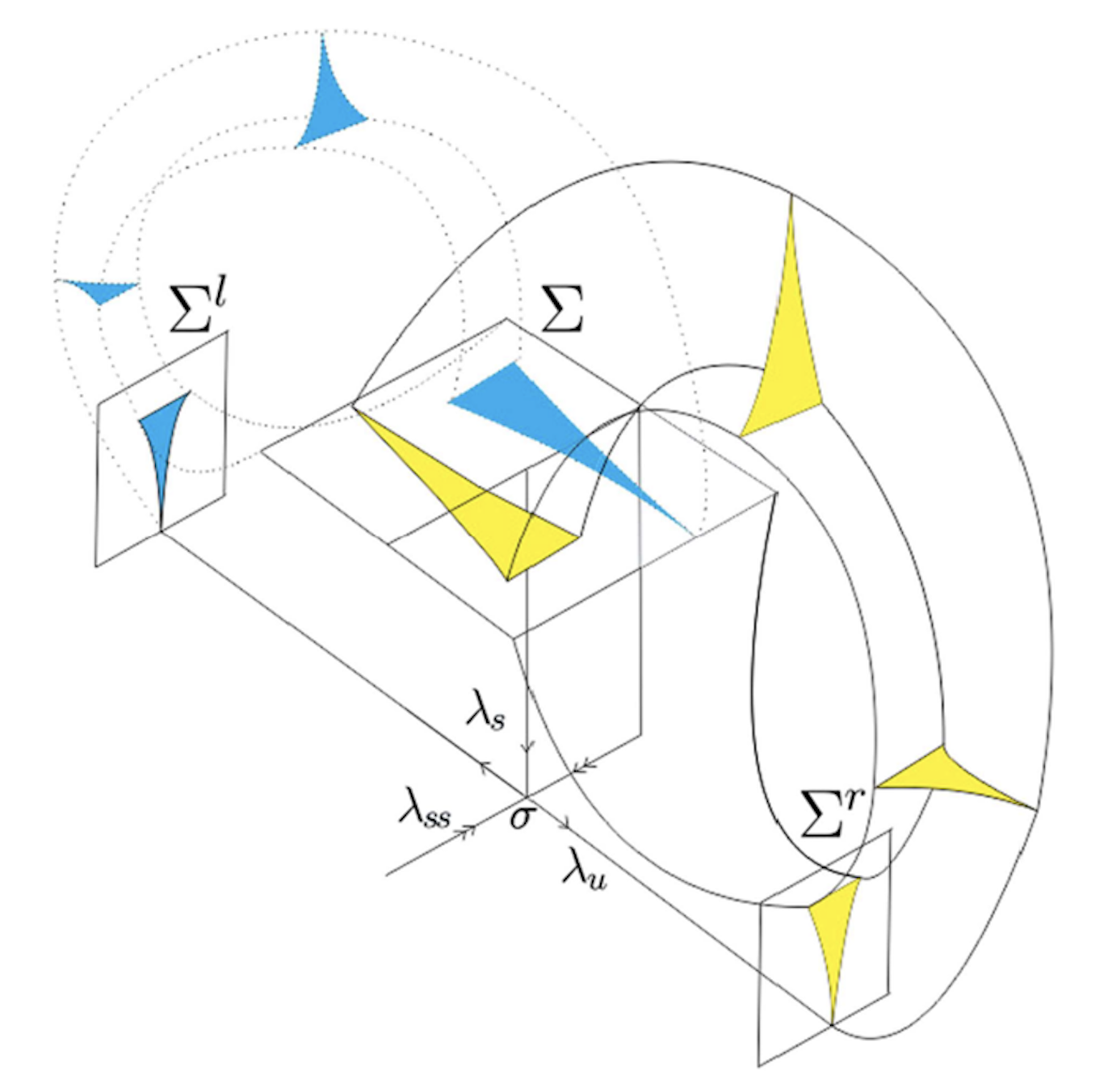}
        \caption{The Geometric Rovella's Attractor:
$$
\lambda_{ss}<\lambda_s<0<\lambda_u
$$
$$
\lambda_s+\lambda_u<0
$$}
        \label{fig:rovella2}
    \end{minipage}
\end{figure}

\begin{definition}
 A compact and invariant set $\Lambda\subset M$ is \emph{singular-hyperbolic} if it is partially hyperbolic with splitting $T_\Lambda M = E^s \oplus E^{cu}$ and the center-unstable bundle $E^{cu}$ is volume expanding; that is, there exist constants $C>0$ and $\lambda>0$ such that
    $$
    \left|\det D X_t|_{E^{cu}(x)}\right| \ge C e^{\lambda t}, \quad \forall x\in \Lambda,\ \forall t\ge 0.
    $$
\end{definition}

It is well known that attracting singular-hyperbolic sets admits a finite union $S$ of local cross-sections which intersects all orbits apart from the singularities, and a Poincar\'e map $\mathcal R:S\to S$ with a contracting foliation. 
So, one can quotient $\mathcal R$ along the stable foliation to obtain a
one-dimensional piecewise $C^{1+\alpha}$ expanding map
$ 
f:I\to I,
$ 
see \cite[Section~4.2.1]{AP10}. These maps admit an absolutely continuous invariant SRB measure \cite{APPV09} and, in view of \cite{ALP}, can be modeled by a full-branch Gibbs--Markov induced map. 
 Another important feature of singular-hyperbolic sets is their robustness: if $\Lambda$ is a  singular-hyperbolic attracting set, then there exist a neighborhood $U$ of $\Lambda$ and an open $C^1$-neighborhood $\mathcal{U}$ of $X$ such that for every $Y\in \mathcal{U}$, the maximal invariant set of $U$ with respect to $Y$ is also a singular-hyperbolic attracting set. In particular, the Poincar\'e map $\mathcal R$ of $\Lambda$ admits a continuation $\mathcal R_Y$ for every vector field sufficiently close to $X$. Under the light of these observations, our first application is as follows.

\begin{theorem}\label{thm:app-sing-hyp}
Let $X$ be a \color{blue} $C^{2}$ \color{black} vector field on $M$. If $\Lambda$ is a singular-hyperbolic attracting set, then $(X,\Lambda)$ satisfies the requirements of Theorems~\ref{thm:dimension-flows}-\ref{thm:intermediate-dimensions}.
\end{theorem}

The previous result applies directly to the geometric Lorenz attractor model, thereby extending the results in \cite{MPR} to a more general setting, while strengthening them to obtain robustness and intermediate dimension properties. We also note that  
Morales, Pac\'ifico and Pujals \cite{MPP04} proved that singular hyperbolic-attractors are exactly the class of robustly transitive attractors of three-dimensional flows.
Hence,  as a direct consequence we obtain the following result:

\begin{corollary}\label{cor:robust-transitive}
If $\Lambda$ is a $C^1$-robustly transitive attractor for a $C^{2}$ three-dimensional vector field $X$, then the following holds:
\begin{enumerate}
    \item[(1)] $(X_t)_t$ has horseshoes of Hausdorff dimension
    strictly larger than $2$;
    \item[(2)] 
    $X$ is $C^2$ approximated by $C^2$-open sets of vector fields $Y$ for which the dynamical
    Lagrange and Markov spectra have non-empty interior, that is,
$$\operatorname{int}\big(\Lspec(({\widetilde Y}_t)_t,\Lambda_{\widetilde Y},\phi)\big)\neq \emptyset \quad \textrm{ and }\quad  \operatorname{int}\big(\Mspec((\widetilde Y_t)_t,\Lambda_{\widetilde Y},\phi)\big)\neq \emptyset$$
for a $C^1$-open and dense set of
    observables $\phi$;
    \item[(3)] 
    for every $c\in [1,\gamma_\Lambda)$ there is an ergodic measure $\mu_c\in \Mer{X}{\Lambda}$ such that $\dim_H(\mu_c)=c$;
    \item[(4)] for every $c\in [1,\gamma_\Lambda^s)$ and every $c'\in [1,\gamma_\Lambda^u)$ there are strictly ergodic hyperbolic sets $K_c,K_{c'}\subset \Lambda$ such that
     $\dims(K_{c})=c $ and $\dimu(K_{c'})=c'.$ 
\end{enumerate}
\end{corollary}

\subsection{Rovella attractors}

 In \cite{Ro93}, Rovella introduced a new class of a geometric attractors, nowadays known as Rovella attractors or contracting Lorenz attractors.
 The construction follows the same lines as the geometric Lorenz attractor, but with a modification in the eigenvalues of the singularity, now yielding three real eigenvalues
$ 
\lambda_{ss} < \lambda_s < 0 < \lambda_u
$  satisfying the dissipativity condition
$\lambda_s + \lambda_u < 0$
(see Figure~\ref{fig:rovella2} for a geometric representation).

More precisely, Rovella considers a one-parameter family
$(X_a)_a$ of $C^3$ vector fields, with corresponding quotient maps
$(f_a)_a$, and proves that there exists a set of parameters
$\mathbb A$ having $0$ as a full Lebesgue density point for which
the critical values of $f_a$ have exponential growth of derivatives
and satisfy a slow recurrence condition to the critical point.
For every $a\in \mathbb A$, the corresponding flow has a transitive
non-hyperbolic attractor $\Lambda_a$, which we refer to as a
Rovella attractor.
While the resulting attractor $\Lambda_a$ is partially hyperbolic with a splitting $T_{\Lambda_a} M = E^s \oplus E^{cu}$,  the Jacobian 
restricted to the center-unstable bundle at the singularity satisfies the condition
$$
\left|\det D X_t|_{E^{cu}(\sigma)}\right| < 1
$$
for $t>0$ sufficiently large, hence it contracts the area in the $E^{cu}$-direction. 
By construction, the Rovella attractor admits a global cross-section  and a smooth stable foliation, and induces a one-dimensional map $f:I \to I$
often called \emph{Rovella maps} or \emph{contracting Lorenz maps}, that we now recall. 

\begin{definition}
\label{def:rov}
A $C^{1+\alpha}$ map $f:[-1,1] \to [-1,1]$ is said to be a \emph{contracting Lorenz map} if it satisfies the following conditions:
\begin{enumerate}
    \item $f$ is strictly monotone on each side of $0$, with limits
    $f(0^-)=1$ and $f(0^+)=-1$;
    \item the maxima of the derivative on each side occur at the endpoints:
    $$
    \max_{x>0} f'(x) = f'(1) \quad\text{and}\quad \max_{x<0} f'(x) = f'(-1);
    $$
    \item the points $1$ and $-1$ are pre-periodic repellers;
    \item the singularity is non-flat and contracting, meaning that 
     $|f'(x)| \sim |x|^{\beta-1}$  as $x\to 0$,
    for some $\beta >1$ (in particular $|f'(x)|\to 0$ as $x\to 0$);
    \item $f$ has negative Schwarzian derivative.
\end{enumerate}
\end{definition}
Despite its similarity with the geometric Lorenz attractor, the analysis of the Rovella model is substantially more challenging, due to the lack of uniform expansion.  
One of the most remarkable characteristics of this attractor is that while it is fragile under $C^1$-perturbations, yet its structure persists for a large set of parameter values (cf. \cite{Ro93} for the precise statement).
In fact, in view of condition (4) at Definition~\ref{def:rov} this class of flows are often regarded as a counterpart of the celebrated quadratic family in the setting of singular flows and the techniques employed in \cite{MPR} do not apply to this setting. 
Depite the quotient map not being uniformly expanding, for
Rovella parameters the maps $f_a$ are non-uniformly expanding and
admit Gibbs--Markov inducing schemes (see \cite{AlvesSoufi}).
Consequently, they fall within the non-uniformly expanding setting
considered in this paper.

\begin{theorem}\label{thm:app-rovella}
Let $a\in\mathbb A$ be a Rovella parameter and let $\Lambda_a$ be
the corresponding contracting Lorenz attractor of the vector field
$X_a$. Then the conclusions of Theorems~A--D hold for
$(X_a,\Lambda_a)$.
\end{theorem}

We emphasize that although the Rovella attractor itself is not robust this causes no obstruction to Theorem~\ref{cor:Markov-spectrum}, since the Gibbs--Markov structure is
used only to produce the initial horseshoe and the robust conclusion
then follows from the hyperbolic continuation of such horseshoe (recall the proof of Theorem~\ref{cor:Markov-spectrum}).
\color{black}

\subsection{General skew product suspensions}

The geometric Lorenz and Rovella attractors provide a natural examples of a broader class of non-uniformly hyperbolic flows that can be described through skew-product Poincar\'e maps, that we now describe. 

We say
that an attracting set $\Lambda$ of a three-dimensional vector field
$X$ is a \emph{skew-product suspension} if it admits a global two-dimensional 
cross-section $S$ for which the Poincar\'e return map can be written,
in suitable coordinates, as a skew-product 
\begin{equation}
\label{eqskew}
R(x,y)=\bigl(f(x),g(x,y)\bigr),    
\end{equation}
where the fiber maps $g(x,\cdot)$ are uniformly contracting and the
quotient map $f$ is a $C^{1+\alpha}$ non-uniformly expanding one-dimensional map.
The dynamics of the flow on $\Lambda$ is then obtained by suspending
the Poincar\'e map $R$ with a corresponding return-time function
$ 
\tau:S\longrightarrow(0,\infty)
$ 
and, in local coordinates is given by
$$
X_t((x,y),s)=\Big(R^n(x,y), t+s-\sum_{j=0}^n\tau(R^j(x,y))\Big)
$$
where $n=n(x,y) \ge 0$ is uniquely defined by
$\sum_{j=0}^n\tau(R^j(x,y))\le t+s< \sum_{j=0}^{n+1}\tau(R^j(x,y))$.
Since the fiber maps are uniformly contracting, the attracting set $\Lambda$ is partially hyperbolic.
Furthermore, if the quotient map $f$ admits a full-branch uniformly expanding
Gibbs--Markov induced map, the main  results of the paper apply to the corresponding attracting set.

\begin{theorem}\label{thm:skew-product-suspension}
Let $\Lambda$ be a skew-product suspension of a
$C^{1+\alpha}$ three-dimensional vector field $X$ with a Poincar\'e map $R$ of the form \eqref{eqskew}. Assume that the
one-dimensional quotient map
$ 
f:I\to I
$ 
admits a full-branch uniformly expanding Gibbs--Markov induced map. Then the conclusions of
Theorems~\ref{thm:dimension-flows}-\ref{thm:intermediate-metric-dimensions} hold for $(X,\Lambda)$.
If, in addition, $X$ is of class $C^2$, then the conclusion of
Theorem~\ref{cor:Markov-spectrum} also holds.
\end{theorem}

Theorem~\ref{thm:skew-product-suspension} allows to consider very broad classes of flows to which our results apply, and  we now briefly recall two of such classes.

\subsubsection{Intermittent Maps}

A standard class of non-uniformly expanding interval maps is given by the
Pomeau--Manneville family
$$
f_b(x)=
\begin{cases}
x\left(1+2^b x^b\right), & x\in\left[0,\frac12\right),\\[1mm]
2x-1, & x\in\left[\frac12,1\right],
\end{cases}
\qquad b>0.
$$
These maps have a neutral fixed point at the origin, meaning that
$$
f_b(0)=0,
\quad\text{and}\quad
f_b'(0)=1,
$$
and are uniformly expanding elsewhere. For $0<b<1$, they admit a
finite absolutely continuous invariant probability measure and a
full-branch uniformly expanding Gibbs--Markov induced map \cite{Young99}.

A second class of non-uniformly expanding interval maps was considered
by Coates, Luzzatto and Muhammad \cite{CLM23}. These are full-branch
maps with two indifferent fixed points, allowing also critical points
and/or singularities with unbounded derivative. More precisely, the
class includes maps
$ 
g:I\to I
$ 
with two monotone full branches and with non-uniform behavior both near
the endpoints and near the common boundary of the branches.
The main feature relevant here is that every map in the class considered
in \cite{CLM23} admits, through a double inducing procedure, a first
return map $G$ 
which is full-branch, uniformly expanding and Gibbs--Markov, with
bounded distortion. Therefore, skew-product suspensions whose quotient
map belongs to this class also satisfy the hypotheses of
Theorem~\ref{thm:skew-product-suspension}.

In particular, Theorems~\ref{thm:dimension-flows}--\ref{thm:intermediate-dimensions}
apply to intermittent flows, that is, attracting sets $\Lambda$ of
$C^{1+\alpha}$ vector fields whose associated skew-product suspension
has a quotient map belonging to one of the intermittent classes
described above.

\subsubsection{Benedicks--Carleson maps}

A further classical class of non-uniformly expanding interval maps is
given by the Benedicks--Carleson quadratic family
$$
f_a(x)=1-ax^2,
\qquad x\in[-1,1],\quad a\in(0,2].
$$
There exists a positive Lebesgue measure set of parameters
${\mathbb A}\subset[a^*,2],$ 
with $a^*<2$ sufficiently close to $2$, such that $2$ is a full
one-sided density point of ${\mathbb A}$ and, for every $a\in{\mathbb A}$, the
critical orbit satisfies suitable expansion and recurrence conditions.
In particular, these maps are non-uniformly expanding despite the
presence of the non-flat critical point at the origin.
For every Benedicks--Carleson parameter $a\in{\mathbb A}$, the map $f_a$
admits a full-branch uniformly expanding Gibbs--Markov induced map
with bounded distortion, as a consequence of \cite{ALP}.
Consequently,
Theorems~\ref{thm:dimension-flows}--\ref{thm:intermediate-dimensions}
apply to Benedicks--Carleson flows, that is, attracting sets $\Lambda$
of $C^{1+\alpha}$ vector fields whose associated skew-product
suspension has a quotient map given by a Benedicks--Carleson map.
\newpage
\section*{Summary of notation}

To help the reader to to follow arguments,  in this section we present a table with the basic notation we use in our arguments. Because our arguments pass to three distinct levels of systems, we believe this table can help the reader to follow more easily the arguments.  
\medskip

\begin{tabularx}{\linewidth}{@{}>{$}p{3.5cm}<{$} X@{}}
\toprule
\multicolumn{1}{@{}l}{\textbf{Symbol}} & \textbf{Meaning}\\
\midrule
\Lambda & attracting set of $X$; used for no other object\\
K,\ \widehat K & a horseshoe of the flow; its trace $K\cap S$ on the cross-section\\
S,\ \mathcal R,\ \tau & global cross-section, Poincar\'e return map, first return time\\
\mathcal F^{s},\ \pi^{s} & stable foliation on $S$ and the quotient projection along it\\
I,\ f & leaf space and quotient map, $f\circ\pi^{s}=\pi^{s}\circ\mathcal R$\\
J,\ R,\ F=f^{R} & domain of the induced map, inducing time, induced map\\
\mathcal P=\{P_i\}_{i\in\mathcal A},\ J_\omega & Markov partition of $J$, countable alphabet, cylinders\\
\Omega,\ \mathcal T,\ \pi & symbolic space, shift map, coding $\pi:\Omega\to\widehat K$\\
\psi^{u}_{X},\ \psi^{u}_{\mathcal R},\ \widetilde\psi^{\,u} & unstable geometric potential: for the flow, for $\mathcal R$, on $\Omega$\\
P_{Z}(\cdot,\cdot),\ \mathcal Z_k(t) & topological pressure on a set $Z$; partition function\\
\delta^{s}(\cdot),\ \delta^{u}(\cdot) & transverse stable and unstable dimensions\\
\gamma^{s},\gamma^{u},\gamma & suprema of $\dims,\dimu,\dim_H$ over ergodic measures on horseshoes\\
\widehat\gamma^{\,s},\widehat\gamma^{\,u},\widehat\gamma & the corresponding suprema on the cross-section\\
\Xi & the suspension operator on measures, $\mathcal M_{\mathcal R}(\widehat K)\to\mathcal M_{X}(K)$\\
\mathcal M_{X},\ \mathcal M_{\mathcal R},\ \mathcal M_{\mathcal T} & invariant probabilities for the flow, for $\mathcal R$, for $\mathcal T$; superscript $e$: ergodic\\
\mathfrak H(\Lambda) & the family of horseshoes contained in $\Lambda$\\
\mathfrak L(\cdot),\ \mathfrak M(\cdot) & Lagrange and Markov dynamical spectra\\
\widetilde Y,\ Y & the perturbed vector field; a vector field near it\\
\mathcal U,\ \mathcal W,\ V & neighbourhoods of vector fields, of diffeomorphisms, of $\widehat K$ in $S$\\
\varphi:N\to N & abstract $C^{2}$ surface diffeomorphism and its ambient surface\\
\lambda^{s},\lambda^{u}\ \ \text{vs}\ \ \lambda_{ss},\lambda_s,\lambda_u & Lyapunov exponents versus eigenvalues at a singularity\\
\bottomrule
\end{tabularx}

\color{black}

\vspace{1cm}

\subsection*{Acknowledgements}
PV was partially supported by CIDMA under the Portuguese Foundation for Science and Technology  (FCT, https://ror.org/00snfqn58)  Multi-Annual Financing Program for R\&D Units
 https://doi.org/10.54499/UID/04106/2025, and by Agence Nationale de la Recherche (ANR) under the ANR 2020 funding programme, project THERMOGAMAS.


\begin{thebibliography}{99}

\bibitem{ABS} V.~S. Afraimovich, V.~V. Bykov and L.~P. Shilnikov, \emph{The origin and structure of the Lorenz attractor}, Dokl. Akad. Nauk SSSR \textbf{234} (1977), 336--339.

\bibitem{AP87} V.~S. Afraimovich and Ya.~B. Pesin, \emph{Dimension of Lorenz type attractors}, in: \emph{Mathematical Physics Reviews, Vol.~6}, Soviet Sci. Rev. Sect. C Math. Phys. Rev. \textbf{6}, Harwood Academic Publ., Chur, 1987, pp.~169--241. 

\bibitem{Alvesbook} J.~F. Alves, \emph{Nonuniformly Hyperbolic Attractors: Geometric and Probabilistic Aspects}, Springer Monogr. Math., Springer, 2020.

\bibitem{AAPP} J.~F. Alves, V. Ara\'ujo, M.~J. Pac\'ifico and V. Pinheiro, \emph{On the volume of singular-hyperbolic sets}, Dyn. Syst. \textbf{22} (2007), 249--267.



\bibitem{ALP} J. F. Alves, S. Luzzatto and V. Pinheiro, Markov structures and decay of correlations for non-uniformly expanding dynamical systems, Ann. Inst. H. Poincaré Anal. Non Linéaire 22 (2005), 817–839.






\bibitem{AlvesSoufi}
J.~F. Alves and M. Soufi,
\emph{Statistical stability and limit laws for Rovella maps},
Nonlinearity \textbf{25} (2012), no.~12, 3527--3552.

\bibitem{AraButVar}
V. Ara{\'u}jo, O. Butterley and P. Varandas,
\textit{Open sets of Axiom A flows with exponentially mixing attractors},
Proc. Amer. Math. Soc. \textbf{144} (2016), no.~7, 2971--2984.





\bibitem{AM17} V. Ara\'ujo and I. Melbourne, \emph{Existence and smoothness of the stable foliation for sectional hyperbolic attractors}, Bull. Lond. Math. Soc. \textbf{49} (2017), 351--367.

\bibitem{AM19} V. Ara\'ujo and I. Melbourne, \emph{Mixing properties and statistical limit theorems for singular hyperbolic flows without a smooth stable foliation}, Adv. Math. \textbf{349} (2019), 212--245.


\bibitem{AP10} V. Ara\'ujo and M.~J. Pac\'ifico, \emph{Three-Dimensional Flows}, Ergeb. Math. Grenzgeb. (3) \textbf{53}, Springer, 2010.


\bibitem{APPV09}
V. Ara\'ujo, M.~J. Pac\'ifico, E. Pujals and M. Viana,
\emph{Singular-hyperbolic attractors are chaotic},
Trans. Amer. Math. Soc. \textbf{361} (2009), no.~5, 2431--2485.


\bibitem{AOR25} A. Arbieto, P. Oprocha and E. Rego, \emph{Entropy flexibility of dynamical systems}, preprint, arXiv:2507.11048.

\bibitem{BaBook} L. Barreira, \emph{Dimension Theory of Hyperbolic Flows}, Springer Monogr. Math., Springer, 2013.


\bibitem{BH} L. Barreira and C. Holanda, \emph{Dimension spectra for flows: future and past}, Nonlinear Anal. Real World Appl. \textbf{65} (2022), Paper No. 103497.

\bibitem{BPS99} L. Barreira, Ya. Pesin and J. Schmeling, \emph{Dimension and product structure of hyperbolic measures}, Ann. of Math. (2) \textbf{149} (1999), 755--783.

\bibitem{BRW04} L. Barreira, L. Radu and C. Wolf, \emph{Dimension of measures for suspension flows}, Dyn. Syst. \textbf{19} (2004), 89--107.

\bibitem{BCR07} F. B\'eguin, S. Crovisier and F. Le~Roux, \emph{Construction of curious minimal uniquely ergodic homeomorphisms on manifolds: the Denjoy--Rees technique}, Ann. Sci. \'Ecole Norm. Sup. (4) \textbf{40} (2007), 251--308.

\bibitem{Bo79} R. Bowen, \emph{Hausdorff dimension of quasi-circles}, Inst. Hautes \'Etudes Sci. Publ. Math. \textbf{50} (1979), 259--273.

\bibitem{BR75} R. Bowen and D. Ruelle, \emph{The ergodic theory of Axiom A flows}, Invent. Math. \textbf{29} (1975), 181--202.

\bibitem{CMM18} A. Cerqueira, C. Matheus and C.~G. Moreira, \emph{Continuity of Hausdorff dimension across generic dynamical Lagrange and Markov spectra}, J. Mod. Dyn. \textbf{12} (2018), 151--174.

\bibitem{CLM23}
D.~Coates, S.~Luzzatto, and M.~Muhammad.
\textit{Doubly intermittent full branch maps with critical points and singularities}.
Comm. Math. Phys. \textbf{402} (2023), no.~2, 1845--1878.

\bibitem{CF89} T.~W. Cusick and M.~E. Flahive, \emph{The Markoff and Lagrange Spectra}, Math. Surveys Monogr. \textbf{30}, Amer. Math. Soc., 1989.

\bibitem{DGRZ25} L.~J. D\'iaz, K. Gelfert, M. Rams and J. Zhang, \emph{Full flexibility of entropies among ergodic measures for partially hyperbolic diffeomorphisms}, preprint, arXiv:2505.03079.

\bibitem{DHT22} Y. Dong, X. Hou and X. Tian, \emph{Abundance of Smale's horseshoes and ergodic measures via multifractal analysis and various quantitative spectrums}, preprint, arXiv:2207.02403.

\bibitem{ER85} J.-P. Eckmann and D. Ruelle, \emph{Ergodic theory of chaos and strange attractors}, Rev. Modern Phys. \textbf{57} (1985), 617--656.

\bibitem{Falconer} K.~J. Falconer, \emph{Fractal Geometry: Mathematical Foundations and Applications}, 2nd ed., John Wiley \& Sons, 2003.


\bibitem{G72} C. Grillenberger, \emph{Constructions of strictly ergodic systems.~I. Given entropy}, Z. Wahrscheinlichkeitstheorie und Verw. Gebiete \textbf{25} (1972/73), 323--334. 

\bibitem{G} J. Guckenheimer, \emph{A strange, strange attractor}, in: \emph{The Hopf Bifurcation and Its Applications}, Appl. Math. Sci. \textbf{19}, Springer, 1976, pp.~368--381. 

\bibitem{HPS} M.~W. Hirsch, C.~C. Pugh and M. Shub, \emph{Invariant manifolds}, Bull. Amer. Math. Soc. \textbf{76} (1970), 1015--1019.

\bibitem{HXZ10} W. Huang, X. Ye and G. Zhang, \emph{Lowering topological entropy over subsets}, Ergodic Theory Dynam. Systems \textbf{30} (2010), 181--209.

\bibitem{Ka80} A. Katok, \emph{Lyapunov exponents, entropy and periodic orbits for diffeomorphisms}, Inst. Hautes \'Etudes Sci. Publ. Math. \textbf{51} (1980), 137--173.

\bibitem{LY85} F. Ledrappier and L.-S. Young, \emph{The metric entropy of diffeomorphisms.~II. Relations between entropy, exponents and dimension}, Ann. of Math. (2) \textbf{122} (1985), 540--574.

\bibitem{LSWW20} M. Li, Y. Shi, S. Wang and X. Wang, \emph{Measures of intermediate entropies for star vector fields}, Israel J. Math. \textbf{240} (2020), 791--819.

\bibitem{LMMR20}
D.~Lima, C.~Matheus, C.~G. Moreira, e S.~Roma\~na.
 \emph{ Classical and Dynamical Markov and Lagrange Spectra: Dynamical, Fractal and Arithmetic Aspects}. World Scientific Publishing Company, 2020.

\bibitem{LM20} D. Lima and C.~G. Moreira, \emph{Phase transitions on the Markov and Lagrange dynamical spectra}, Ann. Inst. H. Poincar\'e C Anal. Non Lin\'eaire \textbf{37} (2020), 1--31.




\bibitem{MM83} H. McCluskey and A. Manning, \emph{Hausdorff dimension for horseshoes}, Ergodic Theory Dynam. Systems \textbf{3} (1983), 251--260.

\bibitem{MPP04} C.~A. Morales, M.~J. Pac\'ifico and E.~R. Pujals, \emph{Robust transitive singular sets for $3$-flows are partially hyperbolic attractors or repellers}, Ann. of Math. (2) \textbf{160} (2004), 375--432.


\bibitem{MPR} C.~G. Moreira, M.~J. Pac\'ifico and S. Roma\~na, \emph{Hausdorff dimension, Lagrange and Markov dynamical spectra for geometric Lorenz attractors}, Bull. Amer. Math. Soc. (N.S.) \textbf{57} (2020), 269--292.

\bibitem{MPV01} C.~G. Moreira, J. Palis and M. Viana, \emph{Homoclinic tangencies and fractal invariants in arbitrary dimension}, C.~R. Acad. Sci. Paris S\'er. I Math. \textbf{333} (2001), 475--480.

\bibitem{MR15} C.~G. Moreira and S. Roma\~na, \emph{On the Lagrange and Markov dynamical spectra for Anosov flows in dimension~$3$}, preprint. 

\bibitem{MR17} C.~G. Moreira and S. Roma\~na, \emph{On the Lagrange and Markov dynamical spectra}, Ergodic Theory Dynam. Systems \textbf{37} (2017), 1570--1591.

\bibitem{MR23} C.~G. Moreira and S. Roma\~na, \emph{On the Lagrange and Markov dynamical spectra for geodesic flows on surfaces with negative curvature}, J. Mod. Dyn. \textbf{19} (2023), 187--236.

\bibitem{MY01} C.~G. Moreira and J.-C. Yoccoz, \emph{Stable intersections of regular Cantor sets with large Hausdorff dimensions}, Ann. of Math. (2) \textbf{154} (2001), 45--96.

\bibitem{O68} V.~I. Oseledets, \emph{A multiplicative ergodic theorem: Lyapunov characteristic numbers for dynamical systems}, Trans. Moscow Math. Soc. \textbf{19} (1968), 197--231.

\bibitem{PT93} J. Palis and F. Takens, \emph{Hyperbolicity and Sensitive Chaotic Dynamics at Homoclinic Bifurcations: Fractal Dimensions and Infinitely Many Attractors}, Cambridge Stud. Adv. Math. \textbf{35}, Cambridge Univ. Press, 1993.

\bibitem{PV} J. Palis and M. Viana, \emph{On the continuity of Hausdorff dimension and limit capacity for horseshoes}, in: \emph{Dynamical Systems (Valpara\'iso, 1986)}, Lecture Notes in Math. \textbf{1331}, Springer, 1988, pp.~150--160.

\bibitem{PY94} J. Palis and J.-C. Yoccoz, \emph{Homoclinic tangencies for hyperbolic sets of large Hausdorff dimension}, Acta Math. \textbf{172} (1994), 91--136.


\bibitem{PP90} W. Parry and M. Pollicott, \emph{Zeta functions and the periodic orbit structure of hyperbolic dynamics,} Ast\'erisque No. 187-188 (1990), 268 pp.


\bibitem{Pe97} Ya. Pesin, \emph{Dimension Theory in Dynamical Systems: Contemporary Views and Applications}, Chicago Lectures in Math., Univ. Chicago Press, 1997.

\bibitem{PS} Ya. Pesin and V. Sadovskaya, \emph{Multifractal analysis of conformal Axiom A flows}, Comm. Math. Phys. \textbf{216} (2001), 277--312.

\bibitem{Ro93} \'A. Rovella, \emph{The dynamics of perturbations of the contracting Lorenz attractor}, Bol. Soc. Brasil. Mat. (N.S.) \textbf{24} (1993), 233--259.


\bibitem{Ru82} D. Ruelle, \emph{Repellers for real analytic maps}, Ergodic Theory Dynam. Systems \textbf{2} (1982), 99--107.


\bibitem{Ste00} T. Steinberger, \emph{Hausdorff dimension of attractors for two-dimensional Lorenz transformations}, Israel J. Math. \textbf{116} (2000), 253--269.


\bibitem{S25} P. Sun, \emph{Ergodic measures of intermediate entropies for dynamical systems with the approximate product property}, Adv. Math. \textbf{465} (2025), Paper No. 110159.

\bibitem{TWW19} X. Tian, S. Wang and X. Wang, \emph{Intermediate Lyapunov exponents for systems with periodic orbit gluing property}, Discrete Contin. Dyn. Syst. \textbf{39} (2019), 1019--1032.

\bibitem{Vi04} D. Viswanath, \emph{The fractal property of the Lorenz attractor}, Phys. D \textbf{190} (2004), 115--128.

\bibitem{Yo82} L.-S. Young, \emph{Dimension, entropy and Lyapunov exponents}, Ergodic Theory Dynam. Systems \textbf{2} (1982), 109--124.

\bibitem{You98} L.-S. Young, \emph{Statistical properties of dynamical systems with some hyperbolicity}, Ann. of Math. (2) \textbf{147} (1998), 585--650.


\bibitem{Young99}
L.-S. Young,
\emph{Recurrence times and rates of mixing},
Israel J. Math. \textbf{110} (1999), 153--188.












\end{thebibliography}
\end{document}